\documentclass[twoside,reqno,11pt]{article}
\usepackage{hyperref} 
\usepackage[T1]{fontenc}
\usepackage{lmodern}
\usepackage[latin1]{inputenc}
\usepackage{latexsym}                 
\usepackage{color}                 
\usepackage{amssymb}
\usepackage{amsmath}
\usepackage{amsthm}
\usepackage{graphicx}
\usepackage{verbatim}
\usepackage{mathptmx}      
\newcommand{\bql}[1]{%
\begin{equation}\label{#1}%
}

\def\filename#1{}

\newcommand{\eq}{\end{equation}}
\def\b1{\mathbf 1}

\newcommand{\R}{\ensuremath{\mathbb{R}}}

\newcommand{\N}{\ensuremath{\mathbb{N}}}

\newcommand{\C}{\ensuremath{\mathbb{C}}}

\def\biglf{\par\bigskip\noindent}
\newcommand{\red}[1]{{\color{red} #1 }}     
\newcommand{\blue}[1]{{\color{blue} #1 }}   %
\newcommand{\green}[1]{{\color{green} #1 }} %

\newtheorem{definition}{Definition}

\newtheorem{remark}{Remark}
\newtheorem{theorem}{Theorem}

\def\biglf{\par\bigskip\noindent}

\begin{document}
\begin{center}
{\bf On the Fractional Derivatives of Radial Basis Functions}
\biglf 
Maryam Mohammadi and Robert Schaback 
\biglf
Version of Feb. 13, 2013
\end{center}
{\bf Abstract:}
The paper provides the fractional integrals and derivatives of the Rie\-mann-Liouville  and Caputo type for the five kinds of radial basis functions (RBFs), including the powers, Gaussian, multiquadric, Matern and thin-plate splines, in one dimension. It allows to use high order numerical methods for solving fractional differential equations. The results are tested by solving two fractional differential equations. The first one is a fractional ODE which is solved by the RBF collocation method and the
second one is a fractional PDE which yesis solved by the method of lines based on the spatial trial spaces
spanned by the Lagrange basis associated to the RBFs.

\medskip
{\it MSC 2010\/}: Primary 65M99, Secondary 35C10 

\smallskip
{\it Key Words and Phrases}: Riemann-Liouville fractional integral,
Riemann-Liouville fractional derivative,
Caputo fractional derivative, Radial basis functions
 \vspace*{-16pt}
\section{Introduction}
Fractional calculus has gained considerable popularity and importance due to its attractive applications
as a new modelling tool in a variety of scientific and engineering fields, such as viscoelasticity
\cite{koeller:1984}, hydrology \cite{benson-et-al}, finance \cite{gorenflo-et-al, raberto-et-al}, and system control \cite{podlubny:1999}. These fractional models, described
in the form of fractional differential equations, tend to be much more appropriate for the description
of memory and hereditary properties of various materials and processes than the traditional integer-order
models. 
In the last decade, a number of numerical methods have been developed to solve fractional differential equations. Most of them rely on 
the finite difference method to discretize both the fractional-order space and time derivative \cite{gao-sun:2011, huang-et-al, lin-xu, meerschaert-tadjeran,  schere-et-al, zhang:2009}. Some numerical schemes using low-order finite elements \cite{fix-roop, roop:2006}, matrix transfer technique \cite{ilic-et-al:2005, ilic-et-al:2006}, and spectral methods \cite{li-xu, li-et-al} have also been proposed.

Unlike traditional numerical methods for solving partial differential equations, meshless methods need
no mesh generation, which is the major problem in finite difference, finite element and spectral methods \cite{mokhtari-mohammdi:2010-1, schaback-wendland:2006-1}.
Radial basis function methods are truly meshless and simple enough to allow modelling of
rather high dimensional problems \cite{frank-schaback:1998-1, frank-schaback:1998-2, hon-et-al:2003-1, kansa:1990-1, mokhtari-mohseni}. These methods can be very efficient numerical schemes to discretize
non-local operators like fractional differential operators. 

In this paper, we provide the required formulas for the fractional integrals and derivatives of Riemann-Liouville and Caputo type for RBFs in one dimension. The rest of the paper is organized as follows. In section 2 we give some important definitions and theorems which are needed throughout the remaining sections of the paper. The corresponding formulas of the fractional integrals and derivatives of Riemann-Liouville and Caputo type for the five kinds of RBFs are given in section 3 . 
The results are applied to solve two fractional differential equations in section 4. The last section
is devoted to a brief conclusion.
\section{Preliminaries}
In this section, we outline some important definitions, theorems and known properties of some special functions used throughout the remaining sections of the paper 
\cite{mathai-haubold,mathworld,  yang-et-al}. In all cases $\alpha$ denotes a non-integer positive order of differentiation and integration. 
\begin{definition} The left-sided Riemann-Liouville fractional integral of order $\alpha$ of function $f(x)$ is defined as
$${_a I_{_x}}^{\!\alpha}f(x)= \frac{1}{\Gamma(\alpha)}\int_a^x (x-\tau)^{\alpha-1} f(\tau)d\tau,\quad x>a.$$
\end{definition}
\begin{definition} The right-sided Riemann-Liouville fractional integral of order $\alpha$ of function $f(x)$ is defined as
$${_x I_{_b}}^{\!\alpha} f(x)= \frac{1}{\Gamma(\alpha)}\int_x^b (\tau-x)^{\alpha-1} f(\tau)dt,\quad  x<b.$$
\end{definition}
\begin{definition} The left-sided Riemann-Liouville fractional derivative of order $\alpha$ of function $f(x)$ is defined as
$${_a D_{_x}}^{\!\alpha} f(x)= \frac{1}{\Gamma(m-\alpha)}\frac{d^m}{dx^m}\int_a^x (x-\tau)^{m-\alpha-1} f(\tau)d\tau,\quad x>a,$$
where $m=\lceil\alpha\rceil.$
\end{definition}
\begin{definition} The right-sided Riemann-Liouville fractional derivative of order $\alpha$ of function $f(x)$ is defined as
$${_x D_{_b}}^{\!\alpha} f(x)= \frac{{(-1)}^m}{\Gamma(m-\alpha)}\frac{d^m}{dx^m}\int_x^b (\tau-x)^{m-\alpha-1} f(\tau)d\tau,\quad x<b,$$
where $m=\lceil\alpha\rceil.$
\end{definition}
\begin{definition} The Riesz space fractional derivative of order $\alpha$ of function $f(x,t)$ on a finite interval $a\leq x\leq b$ is defined as
$$\frac{\partial^\alpha}{\partial{|x|}^\alpha}f(x,t)=-c_\alpha\left({_a D_{_x}}^{\!\alpha}+{_x D_{_b}}^{\!\alpha}\right)f(x,t),$$
where 
\begin{eqnarray*}
&&C_\alpha=\frac{1}{2\cos\left(\frac{\pi\alpha}{2}\right)},\quad\alpha\neq 1,\\
&&{_a D_{_x}}^{\!\alpha} f(x,t)= \frac{1}{\Gamma(m-\alpha)}\frac{d^m}{dx^m}\int_a^x (x-\tau)^{m-\alpha-1} f(\tau,t)d\tau,\\
&&{_x D_{_b}}^{\!\alpha} f(x,t)= \frac{{(-1)}^m}{\Gamma(m-\alpha)}\frac{d^m}{dx^m}\int_x^b (\tau-x)^{m-\alpha-1} f(\tau,t)d\tau.
\end{eqnarray*}
\end{definition}
\begin{definition} The left-sided Caputo fractional derivative of order $\alpha$ of function $f(x)$ is defined as
$${_{_a} {D^{^{\hspace{-4.5mm}C}}}_{_{_x}}}^{\alpha} f(x)= \frac{{1}}{\Gamma(m-\alpha)}\int_a^x (x-\tau)^{m-\alpha-1} f^{(m)}(\tau)d\tau,\quad x>a,$$
where $m=\lceil\alpha\rceil.$
\end{definition}
\begin{definition} The right-sided Caputo fractional derivative of order $\alpha$ of function $f(x)$ is defined as
$${_{_x} {D^{^{\hspace{-4.5mm}C}}}_{_{_b}}}^{\alpha} f(x)= \frac{{{(-1)}^m}}{\Gamma(m-\alpha)}\int_x^b (\tau-x)^{m-\alpha-1} f^{(m)}(\tau)d\tau,\quad x<b,$$
where $m=\lceil\alpha\rceil.$
\end{definition}
The definitions above hold for functions $f$ with special properties depending on the situations.
It is clear that
\begin{eqnarray*}
{_a D_{_x}}^{\!\alpha} f(x)&=&D^m\left[{_a I_{_x}}^{m-\alpha} f(x)\right],\\
{_x D_{_b}}^{\!\alpha} f(x)&=&{(-1)^m}D^m\left[{_x I_{_b}}^{\!m-\alpha} f(x)\right],\\
{_{_a} {D^{^{\hspace{-4.5mm}C}}}_{_{_x}}}^{\alpha} f(x)&=&{_a I_{_x}}^{\!m-\alpha}\left[f^{(m)}(x)\right],\\
{_{_x} {D^{^{\hspace{-4.5mm}C}}}_{_{_b}}}^{\alpha} f(x)&=&{(-1)^m}{_x I_{_b}}^{\!m-\alpha}\left[{f}^{(m)}(x)\right].
\end{eqnarray*}
\begin{theorem}\label{th1}
For $\beta>-1$ and $x>a$ we have $${_a I_{_x}}^{\!\alpha}(x-a)^\beta=\frac{\Gamma(\beta+1)}{\Gamma(\alpha+\beta+1)}{(x-a)}^{\alpha+\beta}.$$
\end{theorem}
\begin{proof}
\begin{eqnarray*}
{_a I_{_x}}^{\!\alpha}(x-a)^\beta&=&
\frac{1}{\Gamma(\alpha)}\int_a^x (x-\tau)^{\alpha-1}(\tau-a)^\beta
d\tau\\
&=&\frac{1}{\Gamma(\alpha)}
\int_0^{x-a} u^{\alpha-1}(x-u-a)^{\beta}du,
\end{eqnarray*}
where $u=x-\tau$.
Now with the change of variable $z=\frac{u}{x-a},$ we get
\begin{eqnarray*}
{_a
  I_{_x}}^{\!\alpha}(x-a)^\beta&=&\frac{{(x-a)}^{\alpha+\beta}}{\Gamma(\alpha)}\int_0^{1}
z^{\alpha-1}{(1-z)}^\beta dz\\
&=&\frac{{(x-a)}^{\alpha+\beta}}{\Gamma(\alpha)}B(\alpha,\beta+1)\\
&=&\frac{\Gamma(\beta+1)}{\Gamma(\alpha+\beta+1)}{(x-a)}^{\alpha+\beta}.
\end{eqnarray*}
\end{proof}
\begin{theorem}\label{th2}
For $\beta>\alpha-1$ and $x>a$ we have $${_a D_{_x}}^{\!\alpha}(x-a)^\beta=\frac{\Gamma(\beta+1)}{\Gamma(\beta-\alpha+1)}{(x-a)}^{\beta-\alpha}.$$
\end{theorem}
\begin{proof}
\begin{eqnarray*}
{_a D_{_x}}^{\!\alpha}(x-a)^\beta&=&D^m\left[{_a
    I_{_x}}^{\!m-\alpha}(x-a)^\beta\right]\\
&=&\frac{\Gamma(\beta+1)}{\Gamma(m-\alpha+\beta+1)}D^m\left[{(x-a)}^{m-\alpha+\beta}\right]\\
&=&\frac{\Gamma(\beta+1)}{\Gamma(\beta-\alpha+1)}{(x-a)}^{\beta-\alpha}.
\end{eqnarray*}
\end{proof}
\begin{theorem}\label{th3}
For $\beta>\alpha-1$ and $x>a$ we have $${_{_a} {D^{^{\hspace{-4.5mm}C}}}_{_{_x}}}^{\alpha}(x-a)^\beta=\frac{\Gamma(\beta+1)}{\Gamma(\beta-\alpha+1)}{(x-a)}^{\beta-\alpha}.$$
\end{theorem}
\begin{proof}
\begin{eqnarray*}
{_{_a} {D^{^{\hspace{-4.5mm}C}}}_{_{_x}}}^{\alpha} (x-a)^\beta&&={_a I_{_x}}^{\!m-\alpha}{\left({\left(x-a\right)}^\beta\right)}^{(m)}=
\frac{\Gamma(\beta+1)}{\Gamma(\beta-m+1)}{_a I_{_x}}^{\!m-\alpha}{(x-a)}^{\beta-m}\\&&=\frac{\Gamma(\beta+1)}{\Gamma(\beta-\alpha+1)}{(x-a)}^{\beta-\alpha}.
\end{eqnarray*}
\end{proof}
\begin{theorem}\label{th-relation}
The following relations between the Riemann-Liouville and the Caputo fractional derivatives hold \cite{ishteva-et-al}:
\begin{eqnarray*}\nonumber
&&{_{_a} {D^{^{\hspace{-4.5mm}C}}}_{_{_x}}}^{\alpha} f(x)={_a D_{_x}}^{\!\alpha}f(x)-\displaystyle\sum_{k=0}^{m-1}\frac{f^{(k)}(a)}{\Gamma(k+1-\alpha)}(x-a)^{k-\alpha},\\
&&{_{_x} {D^{^{\hspace{-4.5mm}C}}}_{_{_b}}}^{\alpha} f(x)={_x D_{_b}}^{\!\alpha}f(x)-\displaystyle\sum_{k=0}^{m-1}\frac{{(-1)}^kf^{(k)}(b)}{\Gamma(k+1-\alpha)}(b-x)^{k-\alpha}.
\end{eqnarray*}
\end{theorem}
We now list some known properties of some special functions,
\begin{eqnarray}
\Gamma(2x)&&\!\!\!\!\!\!\!\!=\frac{2^{2x-1}}{\pi}\Gamma(x)\Gamma(x+\frac{1}{2}),\quad \mbox{for}~\mbox{all}~2x\neq 0,-1,-2,\cdots\label{gam2}\\
b(\alpha,\beta;x)&&\!\!\!\!\!\!\!\!=\frac{x^\alpha}{\alpha}F_{21}(\alpha,1-\beta;\alpha+1;x),\label{bf}\\
F_{21}(a,b;c;x)&&\!\!\!\!\!\!\!\!={(1-x)}^{c-a-b}F_{21}(c-a,c-b;c;x),\label{Frel}\\
\frac{d}{dx}\left(x^\nu K_\nu(x)\right)&&\!\!\!\!\!\!\!\!=-x^\nu K_{\nu-1}(x),\nonumber\\
(x)_{2n}&&\!\!\!\!\!\!\!\!=2^{2n}\left(\frac{x}{2}\right)_{n}\left(\frac{1+x}{2}\right)_{n},\nonumber\\
D^m{x^\lambda}&&\!\!\!\!\!\!\!\!=\frac{\Gamma(\lambda+1)}{\Gamma(\lambda-m+1)}x^{\lambda-m},\quad m\in\mathbb{N},\nonumber
\end{eqnarray}
where $\Gamma(x),$ $b(\alpha,\beta;x),$ $F_{pq}(a_1,\ldots,a_p;b_1,\ldots,b_q;x),$ $K_\nu(x),$ and ${(x)}_{n}$
denote the Gamma function, lower incomplete Beta function, Hypergeometric series, modified Bessel function of the second kind, and  Pochhammer symbol, respectively.
\section{Fractional derivatives of RBFs in one dimension}
Since RBFs are usually evaluated on Euclidean distances,
we have to evaluate $$\mathcal{D}^\alpha\phi(r)=\mathcal{D}^\alpha\phi(|x-y|),\quad \mbox{for}~\mbox{all}~x,y\in\mathbb{R},$$
where $\mathcal{D}^\alpha$ can be one of the notations used for fractional integrals and derivatives in section 2 and $\phi(r)$ is one of the RBFs listed in Table \ref{t1},  \cite{schaback:2009-3}.
The following theorems show that finding the fractional integrals and derivatives of $\phi(x)$ can lead to those of RBFs in one dimension.
\begin{theorem}
For all $x,  y\in\mathbb{R}$ and $x>a$ we have
\begin{eqnarray*}
{_a I_{_x}}^{\!\alpha}\phi(|x-y|)=\xi^\alpha\left({_{\xi(a-y)} I_{_x}}^{\!\alpha}\phi\right)(|x-y|),
\end{eqnarray*}
where $\xi=sign(x-y).$
\end{theorem}
\begin{proof}
We get
\begin{eqnarray*}
{_a I_{_x}}^{\!\alpha}\phi(x-y)
&=&\frac{1}{\Gamma(\alpha)}\int_a^x{(x-\tau)}^{\alpha-1}\phi(\tau-y)d\tau\\
&=&\frac{1}{\Gamma(\alpha)}\int_{a-y}^{x-y}{(x-y-u)}^{\alpha-1}\phi(u)du,
\end{eqnarray*}
where $u=\tau-y.$ Then 
\begin{eqnarray}\label{I+}
{_a I_{_x}}^{\!\alpha}\phi(x-y)=\left({_{(a-y)} I_{_x}}^{\!\alpha}\phi\right)(x-y).
\end{eqnarray}
Moreover,
\begin{eqnarray*}
{_a I_{_x}}^{\!\alpha}\phi(y-x)
&=&\frac{1}{\Gamma(\alpha)}\int_a^x{(x-\tau)}^{\alpha-1}\phi(y-\tau)d\tau\\
&=&\frac{{(-1)}^\alpha}{\Gamma(\alpha)}\int_{y-a}^{y-x}{(y-x-u)}^{\alpha-1}\phi(u)du,
\end{eqnarray*}
where $u=y-\tau.$ Then 
\begin{eqnarray}\label{I-}
{_a I_{_x}}^{\!\alpha}\phi(y-x)={(-1)}^\alpha\left({_{(y-a)} I_{_x}}^{\!\alpha}\phi\right)(y-x).
\end{eqnarray}
Then (\ref{I+}) and (\ref{I-}) give the result.
\end{proof}
\begin{remark}
Similarly, one can show that for all $x,  y\in\mathbb{R}$ and $x<b$ 
\begin{eqnarray*}
{_x I_{_b}}^{\!\alpha}\phi(|x-y|)={(-\xi)}^\alpha\left({_{\xi(b-y)} I_{_x}}^{\!\alpha}\phi\right)(|x-y|),
\end{eqnarray*}
where $\xi=sign(x-y).$
\end{remark}
\begin{theorem}
For all $x,  y\in\mathbb{R}$ and $x>a$ we have
\begin{eqnarray*}
{_a D_{_x}}^{\!\alpha}\phi(|x-y|)=\xi^{-\alpha}\left({_{\xi(a-y)} D_{_x}}^{\!\alpha}\phi\right)(|x-y|),
\end{eqnarray*}
where $\xi=sign(x-y).$
\end{theorem}
\begin{proof}
We get
\begin{equation}\label{D+}
\begin{array}{rcl}
{_a D_{_x}}^{\!\alpha}\phi(x-y)
&=&D^m[{_a I_{_x}}^{\!m-\alpha}\phi(x-y)]\\
&=&D^m[\left({_{(a-y)} I_{_x}}^{\!m-\alpha}\phi\right)(x-y)]\\
&=&\left({_{(a-y)} D_{_x}}^{\!\alpha}\phi\right)(x-y).
\end{array}
\end{equation}
Moreover,
\begin{equation}\label{D-}
\begin{array}{rcl}
{_a D_{_x}}^{\!\alpha}\phi(y-x)
&=&
D^m[{_a I_{_x}}^{\!m-\alpha}\phi(y-x)]\\
&=&D^m[{(-1)}^{m-\alpha}\left({_{(y-a)} I_{_x}}^{\!m-\alpha}\phi\right)(y-x)]\\
&=&{(-1)}^{2m-\alpha}D^m[\left({_{(y-a)} I_{_x}}^{\!m-\alpha}\phi\right)(y-x)]\\
&=&{(-1)}^{-\alpha}\left({_{(y-a)} D_{_x}}^{\!\alpha}\phi\right)(y-x).
\end{array}
\end{equation}
Then (\ref{D+}) and (\ref{D-}) give the result.
\end{proof}
\begin{remark}
Similarly, one can show that for all $x,  y\in\mathbb{R}$ and $x<b$ 
\begin{eqnarray*}
{_x D_{_b}}^{\!\alpha}\phi(|x-y|)={(-\xi)}^{-\alpha}\left({_{\xi(b-y)} D_{_x}}^{\!\alpha}\phi\right)(|x-y|),
\end{eqnarray*}
where $\xi=sign(x-y).$
\end{remark}
\begin{theorem}
For all $x,  y\in\mathbb{R}$ and $x>a$ we have
\begin{eqnarray*}
{_{_a} {D^{^{\hspace{-4.5mm}C}}}_{_{_x}}}^{\alpha}\phi(|x-y|)=\xi^{-\alpha}\left({_{_{\xi(a-y)}}} {D^{^{\hspace{-4.5mm}C}}}_{_{_x}}^{~\alpha}\phi\right)(|x-y|),
\end{eqnarray*}
where $\xi=sign(x-y).$
\end{theorem}
\begin{proof}
We get
\begin{equation}\label{DC+}
\begin{array}{rcl}
{_{_a} {D^{^{\hspace{-4.5mm}C}}}_{_{_x}}}^{\alpha}\phi(x-y)
&=&{_a I_{_x}}^{\!m-\alpha}[(\phi(x-y))^{(m)}]\\
&=&{_a I_{_x}}^{\!m-\alpha}[\phi^{(m)}(x-y)]\\
&=&({_{(a-y)} I_{_x}}^{\!m-\alpha}\phi^{(m)})(x-y)\\
&=&({_{_{(a-y)}} {D^{^{\hspace{-4.5mm}C}}}_{_{_{~x}}}}^{\!\!\alpha}\phi)(x-y).
\end{array} 
\end{equation}
Moreover,
\begin{equation}\label{DC-}
\begin{array}{rcl}
{_{_a} {D^{^{\hspace{-4.5mm}C}}}_{_{_x}}}^{\alpha}\phi(y-x)
&=&{_a I_{_x}}^{\!m-\alpha}[(\phi(y-x))^{(m)}]\\
&=&{(-1)}^m{_a I_{_x}}^{\!m-\alpha}[\phi^{(m)}(y-x)]\\
&=&{(-1)}^{2m-\alpha}({_{(y-a)} I_{_x}}^{\!m-\alpha}\phi^{(m)})(y-x)\\
&=&{(-1)}^{-\alpha}({_{_{(y-a)}} {D^{^{\hspace{-4.5mm}C}}}_{_{_{~x}}}}^{\!\!\alpha}\phi)(y-x).
\end{array} 
\end{equation}
Then (\ref{DC+}) and (\ref{DC-}) give the result.
\end{proof}
\begin{remark}
Similarly, one can show that for all $x,  y\in\mathbb{R}$ and $x<b$ 
\begin{eqnarray*}
{_{_x} {D^{^{\hspace{-4.5mm}C}}}_{_{_{~b}}}}^{\!\alpha}\phi(|x-y|)={(-\xi)}^{-\alpha}\left({_{_{\xi(b-y)}}} {D^{^{\hspace{-4.5mm}C}}}_{_{_x}}^{~\hspace*{1mm}\alpha}\phi\right)(|x-y|),
\end{eqnarray*}
where $\xi=sign(x-y).$
\end{remark}
In the sequel, we evaluate the  Riemann-Liouville fractional integral and derivative, and also the Caputo fractional derivative of $\phi(x)$ corresponding to the five kinds of RBFs listed in Table \ref{t1}. 
\subsection{Powers}
For $\phi(x)=x^\beta$  the following results hold.
\begin{theorem} 
For $x>0$ we have
\begin{eqnarray*}
{_0 I_{_x}}^{\!\alpha}x^\beta&=&\frac{\Gamma(\beta+1)}{\Gamma(\alpha+\beta+1)}x^{\alpha+\beta},\quad \beta>-1\\
{_0 D_{_x}}^{\!\alpha}x^\beta&=&\frac{\Gamma(\beta+1)}{\Gamma(\beta-\alpha+1)}x^{\beta-\alpha},\quad \beta>\alpha-1\\
{_{_0} {D^{^{\hspace{-4.5mm}C}}}_{_{_x}}}^{\alpha}x^\beta&=&\frac{\Gamma(\beta+1)}{\Gamma(\beta-\alpha+1)}x^{\beta-\alpha},\quad \beta>\alpha-1.
\end{eqnarray*}
\end{theorem}
\begin{proof}
Theorems \ref{th1}, \ref{th2} and \ref{th3} for $a=0$ give directly the results.
\end{proof}
\begin{theorem} 
For $a\neq 0,$ $n\in\mathbb{N}$ and $x>a$ we have
\begin{eqnarray*}
{_a I_{_x}}^{\!\alpha}x^n&=&n!{(x-a)}^\alpha\sum_{k=0}^n\frac{a^{n-k}{(x-a)}^k}{(n-k)!\Gamma(\alpha+k+1)},\\
{_a D_{_x}}^{\!\alpha}x^n&=&n!{(x-a)}^{-\alpha}\sum_{k=0}^n\frac{a^{n-k}{(x-a)}^k}{(n-k)!\Gamma(k-\alpha+1)},\\
{_{_a} {D^{^{\hspace{-4.5mm}C}}}_{_{_x}}}^{\alpha}x^n&=&n!a^{-m}{(x-a)}^{m-\alpha}\sum_{k=0}^{n-m}\frac{a^{n-k}{(x-a)}^k}{(n-m-k)!\Gamma(m-\alpha+k+1)}.
\end{eqnarray*}
\end{theorem}
\begin{proof}
The Taylor expansion of $x^n$ about the point $x=a$ gives
\begin{eqnarray*}
x^n=\sum_{k=0}^n\frac{n!a^{n-k}}{(n-k)!k!}{(x-a)}^k.
\end{eqnarray*}
Now, according to the linearity of the Riemann-Liouville fractional integral and derivative, we have
\begin{eqnarray*}
{_a I_{_x}}^{\!\alpha}x^n&=&\sum_{k=0}^n\frac{n!a^{n-k}}{(n-k)!k!}~{_a
  I_{_x}}^{\!\alpha}{(x-a)}^k\\
&=&\sum_{k=0}^n\frac{n!a^{n-k}}{(n-k)!k!}\frac{\Gamma(k+1)}{\Gamma(\alpha+k+1)}{(x-a)}^{\alpha+k}\\
&=&n!{(x-a)}^\alpha\sum_{k=0}^n\frac{a^{n-k}{(x-a)}^k}{(n-k)!\Gamma(\alpha+k+1)}.
\end{eqnarray*}
\begin{eqnarray*}
{_a D_{_x}}^{\!\alpha}x^n&=&\sum_{k=0}^n\frac{n!a^{n-k}}{(n-k)!k!}~{_a
  D_{_x}}^{\!\alpha}{(x-a)}^k\\
&=&\sum_{k=0}^n\frac{n!a^{n-k}}{(n-k)!k!}\frac{\Gamma(k+1)}{\Gamma(k-\alpha+1)}{(x-a)}^{k-\alpha}\\
&=&n!{(x-a)}^{-\alpha}\sum_{k=0}^n\frac{a^{n-k}{(x-a)}^k}{(n-k)!\Gamma(k-\alpha+1)}.
\end{eqnarray*}
\begin{eqnarray*}
{_{_a} {D^{^{\hspace{-4.5mm}C}}}_{_{_x}}}^{\alpha}x^n
&=&{_a I_{_x}}^{\!m-\alpha}{\left({x}^n\right)}^{(m)}=\frac{n!}{(n-m)!}{_a
  I_{_x}}^{\!m-\alpha}x^{n-m}\\
&=&n!a^{-m}{(x-a)}^{m-\alpha}\sum_{k=0}^{n-m}\frac{a^{n-k}{(x-a)}^k}{(n-k)!\Gamma(m-\alpha+k+1)}.
\end{eqnarray*}
\end{proof}
\begin{remark}
Similarly, one can show that for $x<b$ and $n\in\mathbb{N}$
\begin{eqnarray*}
{_x I_{_b}}^{\!\alpha}x^n
&=&
n!{(b-x)}^\alpha\sum_{k=0}^n\frac{b^{n-k}{(x-b)}^k}{(n-k)!\Gamma(\alpha+k+1)},\\
{_x D_{_b}}^{\!\alpha}x^n
&=&
n!{(b-x)}^{-\alpha}\sum_{k=0}^n\frac{b^{n-k}{(x-b)}^k}{(n-k)!\Gamma(k-\alpha+1)},\\
{_{_x} {D^{^{\hspace{-4.5mm}C}}}_{_{_b}}}^{\alpha}x^n
&=&
{(-1)}^mn!b^{-m}{(b-x)}^{m-\alpha}\sum_{k=0}^{n-m}\frac{b^{n-k}{(x-b)}^k}{(n-m-k)!\Gamma(m-\alpha+k+1)}.
\end{eqnarray*}
\end{remark}
\subsection{Gaussian}
For $\phi(x)=\exp(-x^2/2)$ the following results hold.
\begin{theorem} 
For $x>0$ we have
\begin{eqnarray*}
{_0 I_{_x}}^{\!\alpha}\mbox{e}^{-\frac{x^2}{2}}
&=&
\frac{x^\alpha}{\Gamma(1+\alpha)}
F_{22}\left(\frac{1}{2},1;\frac{1+\alpha}{2},\frac{2+\alpha}{2};-\frac{x^2}{2}\right),\\
{_0 D_{_x}}^{\!\alpha}\mbox{e}^{-\frac{x^2}{2}}
&=&\frac{x^{-\alpha}}{\Gamma(1-\alpha)}F_{22}
\left(\frac{1}{2},1;\frac{1-\alpha}{2},\frac{2-\alpha}{2};-\frac{x^2}{2}\right),\\
{_{_0}{D^{^{\hspace{-4.5mm}C}}}_{_{_x}}}^{\alpha}\mbox{e}^{-\frac{x^2}{2}}
&=&\frac{x^{-\alpha}}{\Gamma(1-\alpha)}F_{22}
\left(\frac{1}{2},1;\frac{1-\alpha}{2},\frac{2-\alpha}{2};-\frac{x^2}{2}\right).
\end{eqnarray*}
\end{theorem}
\begin{proof}
The Taylor expansion about the point $x=0$ gives
\begin{eqnarray*}
\mbox{e}^{-\frac{x^2}{2}}=\sum_{n=0}^\infty\frac{{(-1)}^n}{2^nn!}x^{2n}.
\end{eqnarray*}
Therefore 
\begin{eqnarray*}
{_0 I_{_x}}^{\!\alpha}\mbox{e}^{-\frac{x^2}{2}}
&=&\sum_{n=0}^\infty\frac{{(-1)}^n}{2^nn!}{_0 I_{_x}}^{\!\alpha}x^{2n}
=\sum_{n=0}^\infty\frac{{(-1)}^n}{2^nn!}\frac{\Gamma(2n+1)}{\Gamma(2n+\alpha+1)}x^{2n+\alpha}\\
&=&\frac{x^\alpha}{\Gamma(1+\alpha)}\sum_{n=0}^\infty
\frac{{(1)}_{2n}}{{(1+\alpha)}_{2n}n!}{\left(-\frac{x^2}{2}\right)}^n\\
&=&\frac{x^\alpha}{\Gamma(1+\alpha)}\sum_{n=0}^\infty
\frac{{(\frac{1}{2})}_{n}{(1)}_{n}}{{
(\frac{1+\alpha}{2})}_{n}{(\frac{2+\alpha}{2})}_{n}n!}{\left(-\frac{x^2}{2}\right)}^n\\
&=&\frac{x^\alpha}{\Gamma(1+\alpha)}
F_{22}\left(\frac{1}{2},1;\frac{1+\alpha}{2},\frac{2+\alpha}{2};-\frac{x^2}{2}\right).
\end{eqnarray*}
\begin{eqnarray*}
{_0 D_{_x}}^{\!\alpha}\mbox{e}^{-\frac{x^2}{2}}
&=&\sum_{n=0}^\infty\frac{{(-1)}^n}{2^nn!}{_0 D_{_x}}^{\!\alpha}x^{2n}
=\sum_{n=0}^\infty\frac{{(-1)}^n}{2^nn!}\frac{\Gamma(2n+1)}{\Gamma(2n-\alpha+1)}x^{2n-\alpha}\\
&=&\frac{x^{-\alpha}}{\Gamma(1-\alpha)}\sum_{n=0}^\infty\frac{{(1)}_{2n}}{{(1-\alpha)}_{2n}n!}{\left(-\frac{x^2}{2}\right)}^n\\
&=&\frac{x^{-\alpha}}{\Gamma(1-\alpha)}\sum_{n=0}^\infty\frac{{(\frac{1}{2})}_{n}{(1)}_{n}}{{(\frac{1-\alpha}{2})}_{n}{(\frac{2-\alpha}{2})}_{n}n!}{\left(-\frac{x^2}{2}\right)}^n\\
&=&\frac{x^{-\alpha}}{\Gamma(1-\alpha)}
F_{22}\left(\frac{1}{2},1;\frac{1-\alpha}{2},\frac{2-\alpha}{2};-\frac{x^2}{2}\right).
\end{eqnarray*}
\begin{eqnarray*}
{_{_0}
  {D^{^{\hspace{-4.5mm}C}}}_{_{_x}}}^{\alpha}\mbox{e}^{-\frac{x^2}{2}}
&=&\sum_{n=0}^\infty\frac{{(-1)}^n}{2^nn!}{_{_0}
  {D^{^{\hspace{-4.5mm}C}}}_{_{_x}}}^{\alpha}x^{2n}\\
&=&\frac{x^{-\alpha}}{\Gamma(1-\alpha)}F_{22}
\left(\frac{1}{2},1;\frac{1-\alpha}{2},\frac{2-\alpha}{2};-\frac{x^2}{2}\right).
\end{eqnarray*}
\end{proof}
\begin{theorem} 
For $a\neq 0$ and $x>a$ we have 
$$
\begin{array}{rcl}
&&{_a I_{_x}}^{\!\alpha}\mbox{e}^{-\frac{x^2}{2}}\\
&=&
\displaystyle{  {(x-a)}^\alpha\sum_{n=0}^\infty\frac{{(-1)}^n(2n)!}{2^nn!}\left(\sum_{k=0}^{2n}\frac{a^{2n-k}{(x-a)}^k}{(2n-k)!\Gamma(\alpha+k+1)}\right),}\\
&&{_a D_{_x}}^{\!\alpha}\mbox{e}^{-\frac{x^2}{2}}\\
&=&\displaystyle{ {(x-a)}^{-\alpha}\sum_{n=0}^\infty\frac{{(-1)}^n(2n)!}{2^nn!}\left(\sum_{k=0}^{2n}\frac{a^{2n-k}{(x-a)}^k}{(2n-k)!\Gamma(k-\alpha+1)}\right),}\\
&&{a}^{m}{(x-a)}^{\alpha-m}{_{_a}{D^{^{\hspace{-4.5mm}C}}}_{_{_x}}}^{\alpha}\mbox{e}^{-\frac{x^2}{2}}\\
&=&\displaystyle{\sum_{n=0}^\infty\frac{{(-1)}^n(2n)!}{2^nn!}\left(\sum_{k=0}^{2n-m}\frac{a^{2n-k}{(x-a)}^k}{(2n-m-k)!\Gamma(m-\alpha+k+1)}\right).}
\end{array} 
$$
\end{theorem}
\begin{proof}
\begin{eqnarray*}
{_a I_{_x}}^{\!\alpha}\mbox{e}^{-\frac{x^2}{2}}
&=&\sum_{n=0}^\infty\frac{{(-1)}^n}{2^nn!}{_a I_{_x}}^{\!\alpha}x^{2n}\\
&=&\sum_{n=0}^\infty\frac{{(-1)}^n}{2^nn!}(2n)!{(x-a)}^\alpha
\sum_{k=0}^{2n}\frac{a^{2n-k}{(x-a)}^k}{(2n-k)!\Gamma(\alpha+k+1)}\\
&=&{(x-a)}^\alpha\sum_{n=0}^\infty\frac{{(-1)}^n(2n)!}{2^nn!}
\left(\sum_{k=0}^{2n}\frac{a^{2n-k}{(x-a)}^k}{(2n-k)!\Gamma(\alpha+k+1)}\right).
\end{eqnarray*}
\begin{eqnarray*}
{_a D_{_x}}^{\!\alpha}\mbox{e}^{-\frac{x^2}{2}}
&=&\sum_{n=0}^\infty\frac{{(-1)}^n}{2^nn!}{_a D_{_x}}^{\!\alpha}x^{2n}\\
&=&\sum_{n=0}^\infty\frac{{(-1)}^n}{2^nn!}(2n)!{(x-a)}^{-\alpha}\sum_{k=0}^{2n}
\frac{a^{2n-k}{(x-a)}^k}{(2n-k)!\Gamma(k-\alpha+1)}\\
&=&{(x-a)}^{-\alpha}\sum_{n=0}^\infty\frac{{(-1)}^n(2n)!}{2^nn!}
\left(\sum_{k=0}^{2n}\frac{a^{2n-k}{(x-a)}^k}{(2n-k)!\Gamma(k-\alpha+1)}\right).
\end{eqnarray*}
\begin{eqnarray*}
{_{_a}{D^{^{\hspace{-4.5mm}C}}}_{_{_x}}}^{\alpha}\mbox{e}^{-\frac{x^2}{2}}
&=&\sum_{n=0}^\infty\frac{{(-1)}^n}{2^nn!}{_{_a}{D^{^{\hspace{-4.5mm}C}}}_{_{_x}}}^{\alpha}x^{2n}\\
&=&\sum_{n=0}^\infty\frac{{(-1)}^n}{2^nn!}(2n)!a^{-m}{(x-a)}^{m-\alpha}\\
&&\left(\sum_{k=0}^{2n-m}\frac{a^{2n-k}{(x-a)}^k}{(2n-m-k)!\Gamma(m-\alpha+k+1)}\right)\\
&=&{a}^{-m}{(x-a)}^{m-\alpha}\sum_{n=0}^\infty
\frac{{(-1)}^n(2n)!}{2^nn!}\\
&&\left(\sum_{k=0}^{2n-m}
\frac{a^{2n-k}{(x-a)}^k}{(2n-m-k)!\Gamma(m-\alpha+k+1)}\right).
\end{eqnarray*}
\end{proof}
\begin{remark}
Similarly, one can show that for $x<b$
\begin{eqnarray*}
{_x I_{_b}}^{\!\alpha}\mbox{e}^{-\frac{x^2}{2}}
&=&{(b-x)}^\alpha\sum_{n=0}^\infty\frac{{(-1)}^n(2n)!}{2^nn!}
\left(\sum_{k=0}^{2n}\frac{b^{2n-k}{(x-b)}^k}{(2n-k)!\Gamma(\alpha+k+1)}\right),\\
{_x D_{_b}}^{\!\alpha}\mbox{e}^{-\frac{x^2}{2}}
&=&{(b-x)}^{-\alpha}\sum_{n=0}^\infty\frac{{(-1)}^n(2n)!}{2^nn!}
\left(\sum_{k=0}^{2n}\frac{b^{2n-k}{(x-b)}^k}{(2n-k)!\Gamma(k-\alpha+1)}\right),\\
{_{_x}{D^{^{\hspace{-4.5mm}C}}}_{_{_b}}}^{\alpha}\mbox{e}^{-\frac{x^2}{2}}
&=&{(-1)}^m{b}^{-m}{(b-x)}^{m-\alpha}\sum_{n=0}^\infty\frac{{(-1)}^n(2n)!}{2^nn!}\\
&& \left(\sum_{k=0}^{2n-m}\frac{b^{2n-k}{(x-b)}^k}{(2n-m-k)!\Gamma(m-\alpha+k+1)}\right).
\end{eqnarray*}
\end{remark}
\subsection{Multiquadric}
For $\phi(x)={(1+x^2/2)}^{\beta/2},$ $\beta\in\mathbb{R}$ the following results hold.
\begin{theorem} 
For $x>0$ we have 
\begin{eqnarray*}
{_0 I_{_x}}^{\!\alpha}\left(1+\frac{x^2}{2}\right)^\frac{\beta}{2}
&=&\frac{x^\alpha}{\Gamma(1+\alpha)}
F_{32}\left(\frac{1}{2},1,-\frac{\beta}{2};\frac{1+\alpha}{2},\frac{2+\alpha}{2};-\frac{x^2}{2}\right),\\
{_0 D_{_x}}^{\!\alpha}\left(1+\frac{x^2}{2}\right)^\frac{\beta}{2}
&=&\frac{x^{-\alpha}}{\Gamma(1-\alpha)}
F_{22}\left(\frac{1}{2},1,-\frac{\beta}{2};\frac{1-\alpha}{2},\frac{2-\alpha}{2};-\frac{x^2}{2}\right),\\
{_{_0}{D^{^{\hspace{-4.5mm}C}}}_{_{_x}}}^{\alpha}\left(1+\frac{x^2}{2}\right)^\frac{\beta}{2}
&=&\frac{x^{-\alpha}}{\Gamma(1-\alpha)}
F_{22}\left(\frac{1}{2},1,-\frac{\beta}{2};\frac{1-\alpha}{2},\frac{2-\alpha}{2};-\frac{x^2}{2}\right).
\end{eqnarray*}
\end{theorem}
\begin{proof}
The Taylor expansion about the point $x=0$ gives
\begin{eqnarray*}
\left(1+\frac{x^2}{2}\right)^\frac{\beta}{2}=\sum_{n=0}^\infty\frac{{(\frac{\beta}{2})}^{(n)}}{2^nn!}x^{2n}.
\end{eqnarray*}
Therefore 
\begin{eqnarray*}
{_0 I_{_x}}^{\!\alpha}\left(1+\frac{x^2}{2}\right)^\frac{\beta}{2}
&=&\sum_{n=0}^\infty\frac{{(\frac{\beta}{2})}^{(n)}}{2^nn!}{_0 I_{_x}}^{\!\alpha}x^{2n}\\
&=&\sum_{n=0}^\infty\frac{{(-1)}^n{\left(-\frac{\beta}{2}\right)}_{n}}{2^nn!}
\frac{\Gamma(2n+1)}{\Gamma(2n+\alpha+1)}x^{2n+\alpha}\\
&=&\frac{x^\alpha}{\Gamma(1+\alpha)}\sum_{n=0}^\infty
\frac{{(1)}_{2n}{\left(-\frac{\beta}{2}\right)}_{n}}{{(1+\alpha)}_{2n}n!}{\left(-\frac{x^2}{2}\right)}^n\\
&=&\frac{x^\alpha}{\Gamma(1+\alpha)}\sum_{n=0}^\infty
\frac{{(\frac{1}{2})}_{n}{(1)}_{n}{\left(-\frac{\beta}{2}\right)}_{n}}{{(
\frac{1+\alpha}{2})}_{n}{(\frac{2+\alpha}{2})}_{n}n!}{\left(-\frac{x^2}{2}\right)}^n\\
&=&\frac{x^\alpha}{\Gamma(1+\alpha)}
F_{32}\left(\frac{1}{2},1,-\frac{\beta}{2};\frac{1+\alpha}{2},\frac{2+\alpha}{2};-\frac{x^2}{2}\right).
\end{eqnarray*}
\begin{eqnarray*}
{_0 D_{_x}}^{\!\alpha}\left(1+\frac{x^2}{2}\right)^\frac{\beta}{2}
&=&\sum_{n=0}^\infty\frac{{(-1)}^n{\left(-\frac{\beta}{2}\right)}_{n}}{2^nn!}{_0 D_{_x}}^{\!\alpha}x^{2n}\\
&=&\sum_{n=0}^\infty\frac{{(-1)}^n{\left(-\frac{\beta}{2}\right)}_{n}}{2^nn!}
\frac{\Gamma(2n+1)}{\Gamma(2n-\alpha+1)}x^{2n-\alpha}\\
&=&\frac{x^{-\alpha}}{\Gamma(1-\alpha)}\sum_{n=0}^\infty
\frac{{(1)}_{2n}{\left(-\frac{\beta}{2}\right)}_{n}}{{(1-\alpha)}_{2n}n!}{\left(-\frac{x^2}{2}\right)}^n\\
&=&\frac{x^{-\alpha}}{\Gamma(1-\alpha)}\sum_{n=0}^\infty
\frac{{(\frac{1}{2})}_{n}{(1)}_{n}{\left(-\frac{\beta}{2}\right)}_{n}}{{(\frac{1-\alpha}{2})}_{n}{(\frac{2-\alpha}{2})}_{n}n!}{\left(-\frac{x^2}{2}\right)}^n\\
&=&\frac{x^{-\alpha}}{\Gamma(1-\alpha)}
F_{32}\left(\frac{1}{2},1,{-\frac{\beta}{2}};\frac{1-\alpha}{2},\frac{2-\alpha}{2};-\frac{x^2}{2}\right).
\end{eqnarray*}
\begin{eqnarray*}
{_{_0}
  {D^{^{\hspace{-4.5mm}C}}}_{_{_x}}}^{\alpha}\left(1+\frac{x^2}{2}\right)^\frac{\beta}{2}
&=&\sum_{n=0}^\infty\frac{{(-1)}^n{\left(-\frac{\beta}{2}\right)}_{n}}{2^nn!}{_{_0}
  {D^{^{\hspace{-4.5mm}C}}}_{_{_x}}}^{\alpha}x^{2n}\\
&=&\frac{x^{-\alpha}}{\Gamma(1-\alpha)}
F_{32}\left(\frac{1}{2},1,{-\frac{\beta}{2}};\frac{1-\alpha}{2},\frac{2-\alpha}{2};-\frac{x^2}{2}\right).
\end{eqnarray*}
\end{proof}
\begin{theorem} 
For $a\neq 0$ and $x>a$ we have 
\begin{eqnarray*}
{_a I_{_x}}^{\!\alpha}\left(1+\frac{x^2}{2}\right)^\frac{\beta}{2}
&=&\Gamma(1+\beta/2){(x-a)}^\alpha
\sum_{n=0}^\infty\frac{(2n)!}{2^nn!\Gamma(\beta/2-n+1)}\\
&&
\left(\sum_{k=0}^{2n}\frac{a^{2n-k}{(x-a)}^k}{(2n-k)!\Gamma(\alpha+k+1)}\right),\\
{_a D_{_x}}^{\!\alpha}\left(1+\frac{x^2}{2}\right)^\frac{\beta}{2}
&=&\Gamma(1+\beta/2){(x-a)}^{-\alpha}
\sum_{n=0}^\infty\frac{(2n)!}{2^nn!\Gamma(\beta/2-n+1)}\\
&&\left(\sum_{k=0}^{2n}\frac{a^{2n-k}{(x-a)}^k}{(2n-k)!\Gamma(k-\alpha+1)}\right),\\
{_{_a}{D^{^{\hspace{-4.5mm}C}}}_{_{_x}}}^{\alpha}\left(1+\frac{x^2}{2}\right)^\frac{\beta}{2}
&=&\Gamma(1+\beta/2)a^{-m}{(x-a)}^{m-\alpha}
\sum_{n=0}^\infty\frac{(2n)!}{2^nn!\Gamma(\beta/2-n+1)}\\
&&\left(\sum_{k=0}^{2n-m}\frac{a^{2n-k}{(x-a)}^k}{(2n-m-k)!\Gamma(m-\alpha+k+1)}\right).
\end{eqnarray*}
\end{theorem}
\begin{proof}
\begin{eqnarray*}
{_a I_{_x}}^{\!\alpha}\left(1+\frac{x^2}{2}\right)^\frac{\beta}{2}
&=&\sum_{n=0}^\infty\frac{{(\frac{\beta}{2})}^{(n)}}{2^nn!}{_a
  I_{_x}}^{\!\alpha}x^{2n}\\
&=&\sum_{n=0}^\infty\frac{\Gamma(1+\beta/2)}{
\Gamma(\beta/2-n+1)2^nn!}(2n)!{(x-a)}^\alpha\\
&&\left(\sum_{k=0}^{2n}\frac{a^{2n-k}{(x-a)}^k}{(2n-k)!\Gamma(\alpha+k+1)}\right)\\
&=&\Gamma(1+\beta/2){(x-a)}^\alpha
\sum_{n=0}^\infty\frac{(2n)!}{2^nn!\Gamma(\beta/2-n+1)}\\
&&\left(\sum_{k=0}^{2n}\frac{a^{2n-k}{(x-a)}^k}{(2n-k)!\Gamma(\alpha+k+1)}\right).
\end{eqnarray*}
\begin{eqnarray*}
&&{_a D_{_x}}^{\!\alpha}\left(1+\frac{x^2}{2}\right)^\frac{\beta}{2}\\
&=&\sum_{n=0}^\infty\frac{{(\frac{\beta}{2})}^{(n)}}{2^nn!}{_a
  D_{_x}}^{\!\alpha}x^{2n}\\
&=&\sum_{n=0}^\infty\frac{\Gamma(1+\beta/2)}{\Gamma(\beta/2-n+1)2^nn!}(2n)!{(x-a)}^{-\alpha}
\sum_{k=0}^{2n}\frac{a^{2n-k}{(x-a)}^k}{(2n-k)!\Gamma(k-\alpha+1)}\\
&=&\Gamma(1+\beta/2){(x-a)}^{-\alpha}
\sum_{n=0}^\infty\frac{(2n)!}{2^nn!\Gamma(\beta/2-n+1)}\\
&&\left(\sum_{k=0}^{2n}\frac{a^{2n-k}{(x-a)}^k}{(2n-k)!\Gamma(k-\alpha+1)}\right).
\end{eqnarray*}
\begin{eqnarray*}
&&{_{_a}{D^{^{\hspace{-4.5mm}C}}}_{_{_x}}}^{\alpha}\left(1+\frac{x^2}{2}\right)^\frac{\beta}{2}\\
&=&\sum_{n=0}^\infty\frac{{(\frac{\beta}{2})}^{(n)}}{2^nn!}{_{_a}{D^{^{\hspace{-4.5mm}C}}}_{_{_x}}}^{\alpha}x^{2n}\\
&=&\sum_{n=0}^\infty\frac{\Gamma(1+\beta/2)}{\Gamma(\beta/2-n+1)2^nn!}(2n)!a^{-m}{(x-a)}^{m-\alpha}\\
&&\left(\sum_{k=0}^{2n-m}\frac{a^{2n-k}{(x-a)}^k}{(2n-m+k)!\Gamma(m-\alpha+k+1)}\right)\\
&=&\Gamma(1+\beta/2)a^{-m}{(x-a)}^{m-\alpha}
\sum_{n=0}^\infty\frac{(2n)!}{2^nn!\Gamma(\beta/2-n+1)}\\
&&\left(\sum_{k=0}^{2n-m}\frac{a^{2n-k}{(x-a)}^k}{(2n-m-k)!\Gamma(m-\alpha+k+1)}\right).
\end{eqnarray*}
\end{proof}
\begin{remark}
Similarly, one can show that for $x<b$ 
\begin{eqnarray*}
{_x I_{_b}}^{\!\alpha}\left(1+\frac{x^2}{2}\right)^\frac{\beta}{2}
&=&\Gamma(1+\beta/2){(b-x)}^\alpha\sum_{n=0}^\infty\frac{(2n)!}{2^nn!\Gamma(\beta/2-n+1)}\\
&&\left(\sum_{k=0}^{2n}\frac{b^{2n-k}{(x-b)}^k}{(2n-k)!\Gamma(\alpha+k+1)}\right),\\
{_x D_{_b}}^{\!\alpha}\left(1+\frac{x^2}{2}\right)^\frac{\beta}{2}
&=&\Gamma(1+\beta/2){(b-x)}^{-\alpha}\sum_{n=0}^\infty\frac{(2n)!}{2^nn!\Gamma(\beta/2-n+1)}\\
&&\left(\sum_{k=0}^{2n}\frac{b^{2n-k}{(x-b)}^k}{(2n-k)!\Gamma(k-\alpha+1)}\right),
\end{eqnarray*}
\begin{eqnarray*}
&&{_{_x}{D^{^{\hspace{-4.5mm}C}}}_{_{_b}}}^{\alpha}\left(1+\frac{x^2}{2}\right)^\frac{\beta}{2}\\
&=&{(-1)}^m\Gamma(1+\beta/2)b^{-m}{(b-x)}^{m-\alpha}
\sum_{n=0}^\infty\frac{(2n)!}{2^nn!\Gamma(\beta/2-n+1)}\\
&&\left(\sum_{k=0}^{2n-m}\frac{b^{2n-k}{(x-b)}^k}{(2n-m-k)!\Gamma(m-\alpha+k+1)}\right).
\end{eqnarray*}
\end{remark}
\subsection{Thin-plate splines}
For $\phi(x)=x^{2n}\ln(x),$ $n\in\mathbb{N}$ the following results hold.
\begin{theorem}
For $x>0$ we have
\begin{eqnarray*}
&&{_0 I_{_x}}^{\!\alpha}x^{2n}\ln(x)\\
&=&\frac{\Gamma(2n+1)}{\Gamma(2n+1+\alpha)}x^{\alpha+2n}
\left(\ln(x)+\Psi(2n+1)-\Psi(2n+1+\alpha)\right),\\
&&{_{_0}{D^{^{\hspace{-4.5mm}C}}}_{_{_x}}}^{\alpha}x^{2n}\ln(x)\\
&=&\frac{\Gamma(2n+1)}{\Gamma(2n+1-\alpha)}x^{2n-\alpha}
\left(\vphantom{\sum_{r=1}^m}\ln(x)+\Psi(2n-m+1)-\Psi(2n+1-\alpha)
\right.\\
&&\left.+m!\Gamma(2n-m+1)
\sum_{r=1}^m\frac{{(-1)}^{r-1}}{r(m-r)!\Gamma(2n-m+r+1)}\right),\\
&&{_0 D_{_x}}^{\!\alpha}x^{2n}\ln(x)\\
&=&\frac{\Gamma(2n+1)}{\Gamma(2n+1-\alpha)}x^{2n-\alpha}
\left(\vphantom{\sum_{r=1}^m}\ln(x)+\Psi(2n-m+1)-\Psi(2n+1-\alpha)\right.\\
&&\left.+m!\Gamma(2n-m+1)
\sum_{r=1}^m\frac{{(-1)}^{r-1}}{r(m-r)!\Gamma(2n-m+r+1)}\right),
\end{eqnarray*}
where $\Psi(x)$ is the logarithmic derivative of the Gamma function.
\end{theorem}
\begin{proof}
\begin{eqnarray*}
{_a I_{_x}}^{\!\alpha}x^{2n}\ln(x)
&=&\frac{1}{\Gamma(\alpha)}\int_a^x (x-\tau)^{\alpha-1}\tau^{2n}\ln(\tau) d\tau\\
&=&\frac{1}{2}\frac{d}{dn}
\left(\frac{1}{\Gamma(\alpha)}\int_a^x
(x-\tau)^{\alpha-1}\tau^{2n}d\tau\right)\\
&=&\frac{1}{2}\frac{d}{dn}
\left({_a I_{_x}}^{\!\alpha}x^{2n}\right).
\end{eqnarray*}
Then it suffices to find the derivative of the 
Riemann-Liouville fractional integral of the powers RBF $x^{2n}$ with respect to $n$.
Therefore 
\begin{eqnarray*}
{_0 I_{_x}}^{\!\alpha}x^{2n}\ln(x)=\frac{1}{2}\frac{d}{dn}
\left(\frac{\Gamma(2n+1)}{\Gamma(\alpha+2n+1)}x^{\alpha+2n}\right),
\end{eqnarray*}
which in turn gives
\begin{eqnarray*}
&&{_0 I_{_x}}^{\!\alpha}x^{2n}\ln(x)\\
&=&\frac{\left(\Gamma'(2n+1)x^{\alpha+2n}+\Gamma(2n+1)x^{\alpha+2n}\ln(x)\right)%
\Gamma(2n+1+\alpha)}%
{{\left(\Gamma(2n+1+\alpha)\right)}^2}\\
&&
-\frac{\Gamma'(2n+1+\alpha)\Gamma(2n+1)x^{\alpha+2n}}%
{{\left(\Gamma(2n+1+\alpha)\right)}^2}.
\end{eqnarray*}
Now by substituting 
\begin{eqnarray*}
&&\Gamma'(2n+1)=\Psi(2n+1)\Gamma(2n+1),\\
&&\Gamma'(2n+1+\alpha)=\Psi(2n+1+\alpha)\Gamma(2n+1+\alpha),
\end{eqnarray*}
we have
\begin{equation}\label{i0t}
\begin{array}{rcl}
&&{_0 I_{_x}}^{\!\alpha}x^{2n}\ln(x)\\
&=&\displaystyle{ \frac{\Gamma(2n+1)}{\Gamma(2n+1+\alpha)}x^{\alpha+2n}
\left(\ln(x)+\Psi(2n+1)-\Psi(2n+1+\alpha)\right)}.
\end{array} 
\end{equation}
Since it is well known that for the thin-plate 
splines and their derivatives at $x=0$ the 
limiting value $0$ is considered avoiding the singularity,
according to theorem \ref{th-relation} we have
$${_0 D_{_x}}^{\!\alpha}x^{2n}\ln(x)={_{_0} {D^{^{\hspace{-4.5mm}C}}}_{_{_x}}}^{\alpha}x^{2n}\ln(x).$$
Now for the Caputo fractional derivative, we have
\begin{eqnarray*}
{_{_0} {D^{^{\hspace{-4.5mm}C}}}_{_{_x}}}^{\alpha}x^{2n}\ln(x)={_0 I_{_x}}^{\!m-\alpha}\left(x^{2n}\ln(x)\right)^{(m)}.
\end{eqnarray*}
But for $x\neq 0$
\begin{eqnarray*}
\left(x^{2n}\ln(x)\right)^{(m)}
&=&\sum_{r=0}^m
{m \choose r}
\frac{d^{m-r}}{dx^{m-r}}x^{2n}\frac{d^r}{dx^r}\ln(x)\\
&=&\frac{\Gamma(2n+1)}{\Gamma(2n-m+1)}x^{2n-m}\ln(x)\\
&&+\sum_{r=1}^m{m \choose r}
\frac{d^{m-r}}{dx^{m-r}}x^{2n}\frac{d^r}{dx^r}\ln(x)\\
&=&\frac{\Gamma(2n+1)}{\Gamma(2n-m+1)}x^{2n-m}\ln(x)\\
&& +\sum_{r=1}^m{m \choose r}
\frac{\Gamma(2n+1)x^{2n-m+r}}{\Gamma(2n-m+r+1)}\frac{{(-1)}^{r-1}{(r-1)!}}{x^r}\\
&=&\frac{\Gamma(2n+1)}{\Gamma(2n-m+1)}x^{2n-m}\ln(x)\\
&&+m!\Gamma(2n+1)x^{2n-m}\sum_{r=1}^m
\frac{{(-1)}^{r-1}}{r(m-r)!\Gamma(2n-m+r+1)}.
\end{eqnarray*}
Then
\begin{eqnarray*}
&&{_{_0} {D^{^{\hspace{-4.5mm}C}}}_{_{_x}}}^{\alpha}x^{2n}\ln(x)\\
&=&\frac{\Gamma(2n+1)}{\Gamma(2n-m+1)}{_0 I_{_x}}^{\!m-\alpha}x^{2n-m}\ln(x)\\
&&+m!\Gamma(2n+1)
\left(\sum_{r=1}^m\frac{{(-1)}^{r-1}}{r(m-r)!\Gamma(2n-m+r+1)}\right){_0
  I_{_x}}^{\!m-\alpha}x^{2n-m}\\
&=&\frac{\Gamma(2n+1)}{\Gamma(2n+1-\alpha)}x^{2n-\alpha}
\left(\vphantom{\sum_{r=1}^m}\ln(x)+\Psi(2n-m+1)-\Psi(2n+1-\alpha)\right.\\
&&\left. +m!\Gamma(2n-m+1)
\sum_{r=1}^m\frac{{(-1)}^{r-1}}{r(m-r)!\Gamma(2n-m+r+1)}\right).
\end{eqnarray*}
\end{proof}
\begin{theorem} 
For $a\neq 0$ and $x>a$ we have 
\begin{eqnarray*}
&&{_a I_{_x}}^{\!\alpha}x^{2n}\ln(x)\\
&=&\frac{x^{2n+\alpha}{(x-a)}^\alpha}{a^\alpha\Gamma(1+\alpha)}
\left(\vphantom{\sum_{k=0}^{\infty}}\right.\\
&&\left.F_{21}\left(\alpha,\alpha+2n+1;\alpha+1;\frac{a-x}{a}\right)
\left(\ln(x)-\Psi(\alpha+2n+1)\right)\right.\\
&&~~~\left.+\alpha
\displaystyle\sum_{k=0}^{\infty}\frac{{(\alpha+2n+1)_k}\Psi(\alpha+2n+k+1){(a-x)}^k}{{(\alpha+k)}a^kk!}\right),\\
&&{_{_a} {D^{^{\hspace{-4.5mm}C}}}_{_{_x}}}^{\alpha}x^{2n}\ln(x)\\
&=&\frac{\Gamma(2n+1)x^{2n-\alpha}{(x-a)}^{m-\alpha}}{a^{m-\alpha}\Gamma(1+m-\alpha)}
\left(\vphantom{\sum_{k=0}^{\infty}}\right.\\
&&\left.F_{21}\left(m-\alpha,2n-\alpha+1;m-\alpha+1;\frac{a-x}{a}\right)\right.\\
&&\left.\left(\frac{\ln(x)-\Psi(2n-\alpha+1)}{\Gamma(2n-m+1)}+m!
\sum_{r=1}^m\frac{{(-1)}^{r-1}}{r(m-r)!\Gamma(2n-m+r+1)}\right)\right.\\
&&\left.+\frac{(m-\alpha)}{\Gamma(n-m+1)}
\sum_{k=0}^{\infty}\frac{{(2n-\alpha+1)}_k\Psi(2n-\alpha+k+1)}{(m-\alpha+k)a^kk!}{(a-x)}^k\right).
\end{eqnarray*}
\end{theorem}
\begin{proof}
We know that 
\begin{eqnarray*}
{_a I_{_x}}^{\!\alpha}x^\beta=
\frac{{a^\beta(x-a)^\alpha}}{\Gamma(1+\alpha)}F_{21}\left(1,-\beta;\alpha+1;\frac{a-x}{a}\right).
\end{eqnarray*}
Then by using (\ref{Frel}), we get
\begin{eqnarray}\label{nip}
{_a
  I_{_x}}^{\!\alpha}x^\beta=\frac{x^{\alpha+\beta}(x-a)^{\alpha}}{a^\alpha\Gamma(1+\alpha)}
F_{21}\left(\alpha,\alpha+\beta+1;\alpha+1;\frac{a-x}{a}\right).
\end{eqnarray}
Thus
\begin{eqnarray}\label{itps}\nonumber
&&{_a I_{_x}}^{\!\alpha}x^{2n}\ln(x)\\
&=&\frac{1}{2}\frac{d}{dn}
\left(\frac{x^{\alpha+2n}(x-a)^{\alpha}}{a^\alpha\Gamma(1+\alpha)}
F_{21}\left(\alpha,\alpha+2n+1;\alpha+1;\frac{a-x}{a}\right)\right)\nonumber\\
&=&\frac{x^{\alpha}(x-a)^{\alpha}}{2a^\alpha\Gamma(1+\alpha)}\frac{d}{dn}
\left(x^{2n}F_{21}\left(\alpha,\alpha+2n+1;\alpha+1;\frac{a-x}{a}\right)\right)\nonumber\\\nonumber
&=&\frac{x^{\alpha}(x-a)^{\alpha}}{2a^\alpha\Gamma(1+\alpha)}
\left(2x^{2n}\ln(x)F_{21}\left(\alpha,\alpha+2n+1;\alpha+1;\frac{a-x}{a}\right)\right.\\
&&\left.~~+x^{2n}\frac{d}{dn}
\left(F_{21}\left(\alpha,\alpha+2n+1;\alpha+1;\frac{a-x}{a}\right)\right)\right).
\end{eqnarray}
Moreover,
\begin{eqnarray*}
&&\frac{d}{dn}
\left(F_{21}\left(\alpha,\alpha+2n+1;\alpha+1;\frac{a-x}{a}\right)\right)\\
&=&\frac{d}{dn}\left(
\displaystyle\sum_{k=0}^{\infty}\frac{{(\alpha)}_k{(\alpha+2n+1)}_k}%
{{(\alpha+1)}_k}\frac{{(a-x)}^k}{a^kk!}\right),
\end{eqnarray*}
and so
\begin{eqnarray*}
&&\frac{d}{dn}\left(F_{21}\left(\alpha,\alpha+2n+1;\alpha+1;\frac{a-x}{a}\right)\right)\\\nonumber
&=&\displaystyle\sum_{k=0}^{\infty}\frac{{(\alpha)}_k{(a-x)}^k}{{(\alpha+1)}_ka^kk!\Gamma(\alpha+2n+1)^2}\\
&&\left(
2\Gamma'(\alpha+2n+k+1)\Gamma(\alpha+2n+1)\right.\\
&&\left.-2\Gamma'(\alpha+2n+1)\Gamma(\alpha+2n+k+1)\right),\\\nonumber
&=&\displaystyle\sum_{k=0}^{\infty}\frac{2{(\alpha)}_k{(a-x)}^k}{{(\alpha+1)}_ka^kk!\Gamma(\alpha+2n+1)}\\
&&
\left(\Gamma(\alpha+2n+k+1)
\left(\Psi(\alpha+2n+k+1)-\Psi(\alpha+2n+1)\right)\right),\\\nonumber
&=&\displaystyle\sum_{k=0}^{\infty}\frac{2{(\alpha)}_k{(a-x)}^k}{{(\alpha+1)}_ka^kk!}\\
&&\left({(\alpha+2n+1)}_k\left(\Psi(\alpha+2n+k+1)-\Psi(\alpha+2n+1)\right)\right),\\
&=&2\alpha\displaystyle\sum_{k=0}^{\infty}\frac{{(\alpha+2n+1)_k}
\Psi(\alpha+2n+k+1){(a-x)}^k}{(\alpha+k)a^kk!}\\
&&-2\Psi(\alpha+2n+1)F_{21}\left(\alpha,\alpha+2n+1;\alpha+1;\frac{a-x}{a}\right).
\end{eqnarray*}
Then by substituting the relation above into (\ref{itps}) and simplifying the expressions, we obtain 
\begin{eqnarray}\label{iat}\nonumber
&& {_a I_{_x}}^{\!\alpha}x^{2n}\ln(x)\\\nonumber
&=&\frac{x^{2n+\alpha}{(x-a)}^\alpha}{a^\alpha\Gamma(1+\alpha)}
\left(\vphantom{\sum_{k=0}^{\infty}} \right.\\\nonumber
&&\left.F_{21}\left(\alpha,\alpha+2n+1;\alpha+1;\frac{a-x}{a}\right)
\left(\ln(x)-\Psi(\alpha+2n+1)\right)\right.\\
&&\left.+\alpha\displaystyle\sum_{k=0}^{\infty}\frac{{(\alpha+2n+1)_k}
\Psi(\alpha+2n+k+1){(a-x)}^k}{{(\alpha+k)}a^kk!}\right).
\end{eqnarray}
Now for the Caputo fractional derivative, we have
\begin{eqnarray*}
{_{_a} {D^{^{\hspace{-4.5mm}C}}}_{_{_x}}}^{\alpha}x^{2n}\ln(x)={_a I_{_x}}^{\!m-\alpha}\left(x^{2n}\ln(x)\right)^{(m)}.
\end{eqnarray*}
Then by using (\ref{iat}) and (\ref{nip}), and after simplifying the expressions,  we have
\begin{eqnarray*}
&&{_{_a} {D^{^{\hspace{-4.5mm}C}}}_{_{_x}}}^{\alpha}x^{2n}\ln(x)\\
&=&\frac{\Gamma(2n+1)}{\Gamma(2n-m+1)}{_a I_{_x}}^{\!m-\alpha}x^{2n-m}\ln(x)\\
&&+m!\Gamma(2n+1)\left(\sum_{r=1}^m
\frac{{(-1)}^{r-1}}{r(m-r)!\Gamma(2n-m+r+1)}\right){_a
  I_{_x}}^{\!m-\alpha}x^{2n-m}\\
&=&\frac{\Gamma(2n+1)x^{2n-\alpha}{(x-a)}^{m-\alpha}}{a^{m-\alpha}\Gamma(m-\alpha+1)}
\left(\vphantom{\sum_{r=1}^m} \right.\\
&&\left.F_{21}\left(m-\alpha,2n-\alpha+1;m-\alpha+1;\frac{a-x}{a}\right)\right.\\
&&\left.\left(\frac{\ln(x)-\Psi(2n-\alpha+1)}{\Gamma(2n-m+1)}
+m!\sum_{r=1}^m\frac{{(-1)}^{r-1}}{r(m-r)!\Gamma(2n-m+r+1)}\right)\right.\\
&&\left.+\frac{(m-\alpha)}{\Gamma(2n-m+1)}
\sum_{k=0}^{\infty}\frac{{(2n-\alpha+1)}_k\Psi(2n-\alpha+k+1)}{(m-\alpha+k)a^kk!}{(a-x)}^k\right).
\end{eqnarray*}
\end{proof}
\begin{remark}
Similarly, one can show that for $x<b$
\begin{eqnarray*}
&&{_x I_{_b}}^{\!\alpha}x^{2n}\ln(x)\\
&=&
\frac{x^{2n+\alpha}{(b-x)}^\alpha}{b^\alpha\Gamma(1+\alpha)}
\left(\vphantom{\sum_{k=0}^{\infty}} \right.\\
&&\left. F_{21}\left(\alpha,\alpha+2n+1;\alpha+1;\frac{b-x}{b}\right)
\left(\ln(x)-\Psi(\alpha+2n+1)\right)\right.\\
&&~~~\left.+\alpha\displaystyle\sum_{k=0}^{\infty}\frac{{(\alpha+2n+1)_k}
\Psi(\alpha+2n+k+1){(b-x)}^k}{{(\alpha+k)}b^kk!}\right),
\end{eqnarray*}
\begin{eqnarray*}
&&{_{_x}{D^{^{\hspace{-4.5mm}C}}}_{_{_b}}}^{\alpha}x^{2n}\ln(x)\\
&=&\frac{{(-1)}^{m}\Gamma(2n+1)x^{2n-\alpha}{(b-x)}^{m-\alpha}}{b^{m-\alpha}\Gamma(m-\alpha+1)}
\left( \vphantom{\sum_{r=1}^m} \right.\\
&&\left. F_{21}\left(m-\alpha,2n-\alpha+1;m-\alpha+1;\frac{b-x}{b}\right)\right.\\
&&\left.\left(\frac{\ln(x)-\Psi(2n-\alpha+1)}{\Gamma(2n-m+1)}
+m!\sum_{r=1}^m\frac{{(-1)}^{r-1}}{r(m-r)!\Gamma(2n-m+r+1)}\right)\right.\\
&&\left.+\frac{(m-\alpha)}{\Gamma(2n-m+1)}
\sum_{k=0}^{\infty}\frac{{(2n-\alpha+1)}_k\Psi(2n-\alpha+k+1)}{(m-\alpha+k)a^kk!}{(b-x)}^k\right).
\end{eqnarray*}
\end{remark}
\subsection{Matern}
For $\phi(x)=x^\nu K_\nu(x)$ with non-integer $\nu>0,$ the following results hold.
\begin{theorem}
For $x>a$ we have
\begin{eqnarray*}
&& {_a I_{_x}}^{\!\alpha}{x}^\nu K_\nu(x)\\
&=&\frac{\pi}{2\sin\pi\nu}
\left(\frac{2^\nu x^\alpha}{\Gamma(1-\nu)\Gamma(1+\alpha)}
F_{23}\left(\frac{1}{2},1;1-\nu,\frac{\alpha+1}{2},\frac{\alpha+2}{2};\frac{x^2}{4}\right)\right.\\
&&~~-\frac{2^\nu\Gamma(\nu+\frac{1}{2})}{\sqrt{\pi}\Gamma(\alpha+2\nu+1)}x^{\alpha+2\nu}
F_{12}\left(\nu+\frac{1}{2};\frac{\alpha+2\nu+1}{2};\frac{\alpha}{2}+\nu+1;\frac{x^2}{4}\right)\\
&&~~-\frac{a2^\nu x^{\alpha-1}}{\Gamma(\alpha)}
\sum_{k=0}^\infty
\frac{a^{2k}F_{21}\left(2k+1,1-\alpha;2k+2;\frac{a}{x}\right)}{(2k+1)\Gamma(-\nu+k+1)4^kk!}\\
&&~~+\frac{a^{2\nu+1}}{2^\nu\Gamma(\alpha)}x^{\alpha-1}
\sum_{k=0}^\infty\left(\frac{a^{2k}}{4^kk!(2\nu+k+1)\Gamma(\nu+k+1)}\right.\\
&&~~~~\left.\left.F_{21}(2\nu+2k+1,1-\alpha;2\nu+2k+2;\frac{x^2}{4})\right)\right).
\end{eqnarray*}
\end{theorem}
\begin{proof}
We know that
\begin{eqnarray*}
{_a I_{_x}}^{\!\alpha}{x}^\nu K_\nu(x)=\frac{\pi}{2\sin(\pi\nu)}\left({_a I_{_x}}^{\!\alpha}{x}^\nu J_{-\nu}(x)-{_a I_{_x}}^{\!\alpha}{x}^\nu J_{\nu}(x)\right),
\end{eqnarray*}
where
\begin{eqnarray*}
J_{\nu}(x)=\left(\frac{x}{2
}\right)^\nu\sum_{k=0}^\infty\frac{{\left(\frac{x^2}{4}\right)}^k}{k!
\Gamma(\nu+k+1)}.
\end{eqnarray*}
Now
\begin{eqnarray*}
&&{_a I_{_x}}^{\!\alpha}{x}^\nu J_{\nu}(x)\\
&=&\frac{1}{\Gamma(\alpha)}\int_a^x (x-\tau)^{\alpha-1}\tau^{\nu} 
J_{\nu}(\tau) d\tau \\
&=&\frac{1}{\Gamma(\alpha)}\int_a^x
(x-\tau)^{\alpha-1}\tau^{\nu}\left(\frac{\tau}{2	}\right)^\nu
\sum_{k=0}^\infty\frac{{\left(\frac{\tau^2}{4}\right)}^k}{k!\Gamma(\nu+k+1)}d\tau\\
&=&\sum_{k=0}^\infty\frac{1}{\Gamma(\alpha)2^\nu4^kk!\Gamma(\nu+k+1)}\int_a^x(x-\tau)^{\alpha-1}\tau^{2\nu+2k}d\tau.
\end{eqnarray*}
By the change of variable $u=\frac{\tau}{x}$, we have
\begin{eqnarray*}
&&\int_a^x (x-\tau)^{\alpha-1}\tau^{2\nu+2k}d\tau\\
&=&x^{\alpha+2\nu+2k}\int_{\frac{a}{x}}^1{(1-u)}^{\alpha-1}u^{2\nu+2k}du\\
&=&x^{\alpha+2\nu+2k}
\left(\frac{\Gamma(2\nu+2k+1)\Gamma(\alpha)}%
{\Gamma(\alpha+2\nu+2k+1)}-b\left(2\nu+2k+1,\alpha;\frac{a}{x}\right)\right).
\end{eqnarray*}
Therefore
\begin{eqnarray*}
&&{_a I_{_x}}^{\!\alpha}{x}^\nu J_{\nu}(x)\\
&=&\sum_{k=0}^\infty\frac{x^{\alpha+2\nu+2k}}{2^\nu4^kk!\Gamma(\alpha)\Gamma(\nu+k+1)}\\
&&
\left(\frac{\Gamma(2\nu+2k+1)\Gamma(\alpha)}{\Gamma(\alpha+2\nu+2k+1)}
-b\left(2\nu+2k+1,\alpha;\frac{a}{x}\right)\right)\\
&=&\sum_{k=0}^\infty\frac{\Gamma(2\nu+2k)x^{\alpha+2\nu+2k}}%
{2^{\nu-1}4^kk!\Gamma(\nu+k)\Gamma(\alpha+2\nu+2k+1)}\\
&&
-\sum_{k=0}^\infty\frac{b(2\nu+2k+1,\alpha;\frac{a}{x})x^{\alpha+2\nu+2k}}%
{2^\nu4^kk!\Gamma(\alpha)\Gamma(\nu+k+1)}.
\end{eqnarray*}
Now by using (\ref{gam2}) and (\ref{bf}) we get
\begin{eqnarray*}
&&{_a I_{_x}}^{\!\alpha}{x}^\nu J_{\nu}(x)\\
&=&\sum_{k=0}^\infty
\frac{2^\nu\Gamma(\nu+k+\frac{1}{2})x^{\alpha+2\nu+2k}}%
{\sqrt{\pi}\Gamma(\alpha+2\nu+2k+1)k!}\\
&&~~-\frac{a^{2\nu+1}x^{\alpha-1}}{2^\nu\Gamma(\alpha)}
\sum_{k=0}^\infty\frac{a^{2k}F_{21}(2\nu+2k+1,1-\alpha;2\nu+2k+2;\frac{a}{x})}%
{4^kk!(2\nu+2k+1)\Gamma(\nu+k+1)}\\
&=&\frac{x^{\alpha+2\nu}2^\nu}{\sqrt{\pi}}
\sum_{k=0}^\infty\frac{\Gamma(\nu+k+\frac{1}{2})x^{2k}}{\Gamma(\alpha+2\nu+2k+1)k!}\\
&&~~-\frac{a^{2\nu+1}x^{\alpha-1}}{2^\nu\Gamma(\alpha)}
\sum_{k=0}^\infty\frac{a^{2k}F_{21}(2\nu+2k+1,1-\alpha;2\nu+2k+2;\frac{a}{x})}%
{4^kk!(2\nu+2k+1)\Gamma(\nu+k+1)}\\
&=&\frac{x^{\alpha+2\nu}2^\nu\Gamma(\nu+\frac{1}{2})}{\sqrt{\pi}\Gamma(\alpha+2\nu+1)}
\sum_{k=0}^\infty\frac{{(\nu+\frac{1}{2})}_kx^{2k}}{{(\alpha+2\nu+1)}_{2k}k!}\\
&&~~-\frac{a^{2\nu+1}x^{\alpha-1}}{2^\nu\Gamma(\alpha)}
\sum_{k=0}^\infty\frac{a^{2k}F_{21}(2\nu+2k+1,1-\alpha;2\nu+2k+2;\frac{a}{x})}%
{4^kk!(2\nu+2k+1)\Gamma(\nu+k+1)}\\
&=&\frac{x^{\alpha+2\nu}2^\nu\Gamma(\nu+\frac{1}{2})}{\sqrt{\pi}\Gamma(\alpha+2\nu+1)}
\sum_{k=0}^\infty
\frac{{(\nu+\frac{1}{2})}_kx^{2k}}{{(\frac{\alpha}{2}+\nu+\frac{1}{2})}_{k}%
{(\frac{\alpha}{2}+\nu+1)}_k4^{k}k!}\\
&&~~-\frac{a^{2\nu+1}x^{\alpha-1}}{2^\nu\Gamma(\alpha)}
\sum_{k=0}^\infty\frac{a^{2k}F_{21}(2\nu+2k+1,1-\alpha;2\nu+2k+2;\frac{a}{x})}%
{4^kk!(2\nu+2k+1)\Gamma(\nu+k+1)}\\
&=&\frac{2^\nu\Gamma(\nu+\frac{1}{2})}{\sqrt{\pi}\Gamma(\alpha+2\nu+1)}x^{\alpha+2\nu}
F_{12}\left(\nu+\frac{1}{2};\frac{\alpha}{2}+\nu+\frac{1}{2},
\frac{\alpha}{2}+\nu+1;\frac{x^2}{4}\right)\\
&&~~-\frac{a^{2\nu+1}x^{\alpha-1}}{2^\nu\Gamma(\alpha)}
\sum_{k=0}^\infty\frac{a^{2k}F_{21}(2\nu+2k+1,1-\alpha;2\nu+2k+2;\frac{a}{x})}%
{4^kk!(2\nu+2k+1)\Gamma(\nu+k+1)}.
\end{eqnarray*}
Moreover,
\begin{eqnarray*}
&& {_a I_{_x}}^{\!\alpha}{x}^\nu J_{-\nu}(x)\\
&=&\frac{1}{\Gamma(\alpha)}\int_a^x (x-\tau)^{\alpha-1}\tau^{\nu} J_{-\nu}(\tau) d\tau\\
&=&\frac{1}{\Gamma(\alpha)}
\int_a^x (x-\tau)^{\alpha-1}\tau^{\nu}\left(\frac{\tau}{2	}\right)^{-\nu}
\sum_{k=0}^\infty\frac{{\left(\frac{\tau^2}{4}\right)}^k}{k!\Gamma(-\nu+k+1)}d\tau\\
&=&\sum_{k=0}^\infty\frac{2^\nu}{\Gamma(\alpha)4^kk!\Gamma(-\nu+k+1)}\int_a^x(x-\tau)^{\alpha-1}\tau^{2k}d\tau\\
&=&\sum_{k=0}^\infty\frac{2^\nu
  x^{\alpha+2k}}{\Gamma(\alpha)4^kk!\Gamma(-\nu+k+1)}
\left(\frac{\Gamma(2k+1)\Gamma(\alpha)}{\Gamma(\alpha+2k+1)}-b(2k+1,\alpha;\frac{a}{x})\right)\\
&=&2^\nu
x^\alpha\sum_{k=0}^\infty\frac{\Gamma(2k+1)x^{2k}}{4^kk!\Gamma(-\nu+k+1)\Gamma(\alpha+2k+1)}\\
&&-\sum_{k=0}^\infty\frac{2^\nu
  x^{\alpha+2k}b(2k+1,\alpha;\frac{a}{x})}{4^kk!\Gamma(\alpha)\Gamma(-\nu+k+1)}\\
&=&\frac{2^\nu x^\alpha}{\Gamma(1-\nu)\Gamma(1+\alpha)}
\sum_{k=0}^\infty\frac{{(1)}_{2k}x^{2k}}{{(1-\nu)}_{k}{(1+\alpha)}_{2k}4^kk!}\\
&&-\frac{2^\nu x^{\alpha-1}a}{\Gamma(\alpha)}
\sum_{k=0}^\infty\frac{a^{2k}F_{21}(2k+1,1-\alpha;2k+2;\frac{a}{x})}{4^k(2k+1)\Gamma(-\nu+k+1)k!}\\
&&=\frac{2^\nu x^\alpha}{\Gamma(1-\nu)\Gamma(1+\alpha)}
\sum_{k=0}^\infty\frac{{(\frac{1}{2})}_{k}{(1)}_kx^{2k}}{{(1-\nu)}_{k}%
{(\frac{1+\alpha}{2})}_{k}{(\frac{2+\alpha}{2})}_{k}4^kk!}\\
&&-\frac{2^\nu x^{\alpha-1}a}{\Gamma(\alpha)}
\sum_{k=0}^\infty\frac{a^{2k}F_{21}(2k+1,1-\alpha;2k+2;\frac{a}{x})}{4^k(2k+1)\Gamma(-\nu+k+1)k!}\\
&=&\frac{2^\nu x^\alpha}{\Gamma(1-\nu)\Gamma(1+\alpha)}
F_{23}\left(\frac{1}{2},1;1-\nu,\frac{1+\alpha}{2},\frac{2+\alpha}{2};\frac{x^2}{4}\right)\\
&&-\frac{2^\nu
  x^{\alpha-1}a}{\Gamma(\alpha)}\sum_{k=0}^\infty
\frac{a^{2k}F_{21}(2k+1,1-\alpha;2k+2;\frac{a}{x})}{4^k(2k+1)\Gamma(-\nu+k+1)k!}.
\end{eqnarray*}
Thus
\begin{eqnarray*}
&&{_a I_{_x}}^{\!\alpha}{x}^\nu K_\nu(x)\\
&=&\frac{\pi}{2\sin\pi\nu}
\left(\frac{2^\nu x^\alpha}{\Gamma(1-\nu)\Gamma(1+\alpha)}
F_{23}\left(\frac{1}{2},1;1-\nu,\frac{\alpha+1}{2},\frac{\alpha+2}{2};\frac{x^2}{4}\right)\right.\\
&&~~-\frac{2^\nu\Gamma(\nu+\frac{1}{2})}{\sqrt{\pi}\Gamma(\alpha+2\nu+1)}x^{\alpha+2\nu}
F_{12}\left(\nu+\frac{1}{2};\frac{\alpha+2\nu+1}{2};\frac{\alpha}{2}+\nu+1;\frac{x^2}{4}\right)\\
&&~~-\frac{2^\nu x^{\alpha-1}a}{\Gamma(\alpha)}
\sum_{k=0}^\infty\frac{a^{2k}F_{21}(2k+1,1-\alpha;2k+2;\frac{a}{x})}{(2k+1)\Gamma(-\nu+k+1)4^kk!}\\
&&~~+\frac{a^{2\nu+1}}{2^\nu\Gamma(\alpha)}x^{\alpha-1}
\sum_{k=0}^\infty\left(\frac{a^{2k}}{4^kk!(2\nu+k+1)\Gamma(\nu+k+1)}\right.\\
&&~~~~\left.\left.F_{21}\left(2\nu+2k+1,1-\alpha;2\nu+2k+2;\frac{x^2}{4}\right)\right)\right).
\end{eqnarray*} 
\end{proof}
\begin{theorem}
For $x>a$ we have
\begin{eqnarray*}
{_{_a} {D^{^{\hspace{-4.5mm}C}}}_{_{_x}}}^{\alpha}{x}^\nu K_\nu(x)={(-1)}^m{_a I_{_x}}^{\!m-\alpha}{x}^{\nu-m} K_{\nu-m}(x).
\end{eqnarray*}
\end{theorem}
\begin{proof}
By definition, we have
\begin{eqnarray*}
{_{_a} {D^{^{\hspace{-4.5mm}C}}}_{_{_x}}}^{\alpha}{x}^\nu K_\nu(x)&&\!\!\!={_a I_{_x}}^{\!m-\alpha}{({x}^\nu K_\nu(x))}^{(m)}={(-1)}^m{_a I_{_x}}^{\!m-\alpha}{{x}^{\nu-m} K_{\nu-m}(x)}.
\end{eqnarray*}
\end{proof}
\begin{remark}
For the special case $a=0$ and $x>0$ we have
\begin{eqnarray*}
&&{_0 I_{_x}}^{\!\alpha}{x}^\nu K_\nu(x)\\
&=&\frac{\pi}{2\sin\pi\nu}
\left(\frac{2^\nu x^\alpha}{\Gamma(1-\nu)\Gamma(1+\alpha)}
F_{23}\left(\frac{1}{2},1;1-\nu,\frac{\alpha+1}{2},\frac{\alpha+2}{2};\frac{x^2}{4}\right)\right.\\
&&~~\left.-\frac{2^\nu\Gamma(\nu+\frac{1}{2})}{\sqrt{\pi}\Gamma(\alpha+2\nu+1)}x^{\alpha+2\nu}
F_{12}\left(\nu+\frac{1}{2};\frac{\alpha+2\nu+1}{2};\frac{\alpha}{2}+\nu+1;\frac{x^2}{4}\right)\right).
\end{eqnarray*}
\begin{eqnarray*}
{_{_0} {D^{^{\hspace{-4.5mm}C}}}_{_{_x}}}^{\alpha}{x}^\nu K_\nu(x)={(-1)}^m{_0 I_{_x}}^{\!m-\alpha}{x}^{\nu-m} K_{\nu-m}(x).
\end{eqnarray*} 
\end{remark}
\begin{remark}
Similarly, one can show that for $x<b$ 
\begin{eqnarray*}
&&{_x I_{_b}}^{\!\alpha}{x}^\nu K_\nu(x)\\
&=&
\frac{{(-1)}^\alpha\pi}{2\sin\pi\nu}
\left(\frac{2^\nu x^\alpha}{\Gamma(1-\nu)\Gamma(1+\alpha)}
F_{23}(\frac{1}{2},1;1-\nu,\frac{\alpha+1}{2},\frac{\alpha+2}{2};\frac{x^2}{4})\right.\\
&&~~-\frac{2^\nu\Gamma(\nu+\frac{1}{2})}{\sqrt{\pi}\Gamma(\alpha+2\nu+1)}x^{\alpha+2\nu}
F_{12}(\nu+\frac{1}{2};\frac{\alpha+2\nu+1}{2};\frac{\alpha}{2}+\nu+1;\frac{x^2}{4})\\
&&~~-\frac{2^\nu x^{\alpha-1}b}{\Gamma(\alpha)}
\sum_{k=0}^\infty\frac{b^{2k}F_{21}(2k+1,1-\alpha;2k+2;\frac{b}{x})}{(2k+1)\Gamma(-\nu+k+1)4^kk!}\\
&&~~+\frac{b^{2\nu+1}}{2^\nu\Gamma(\alpha)}x^{\alpha-1}
\sum_{k=0}^\infty\left(\frac{b^{2k}}{4^kk!(2\nu+k+1)\Gamma(\nu+k+1)}\right.\\
&&~~~~\left.\left.F_{21}\left(2\nu+2k+1,1-\alpha;2\nu+2k+2;\frac{x^2}{4}\right)\right)\right),
\end{eqnarray*}
\begin{eqnarray*}
{_{_x} {D^{^{\hspace{-4.5mm}C}}}_{_{_b}}}^{\alpha}{x}^\nu K_\nu(x)={(-1)}^{-\alpha}{_{_a} {D^{^{\hspace{-4.5mm}C}}}_{_{_x}}}^{\alpha}{x}^\nu K_\nu(x).
\end{eqnarray*}
\end{remark}
\section{Application }
In this section we apply the results of the previous section to solve two fractional differential equations. The first one is a fractional ODE which is solved by the RBF collocation method and the second one is a fractional PDE which is solved by the method of lines based on the spatial trial spaces spanned by the Lagrange basis associated to the RBFs. In both cases, we work with scaled RBFs, i.e.
\begin{eqnarray*}
\phi\left(\frac{|x-y|}{c}\right),
\end{eqnarray*}
where the RBF scale $c$ controls their flatness. The infinite sums appearing in the previous formulas are truncated once the terms are smaller in magnitude than machine precision.
\subsection{Test problem 1}
Consider the following fractional ODE \cite{podlubny:1999}:
\begin{eqnarray*}
&&{_0 D_{_t}}^{\!\!3/2}u(t)+u(t)=f(t),\quad t\in(0,T],\\
&&u(0)=u'(0)=0.
\end{eqnarray*}
Let $t_i,~1\leq i\leq n$ be the equidistant discretization points in the interval $[0,T]$ such that $t_1=0$ and $t_n=T$. Then the approximate solution can be written as 
\begin{eqnarray*}
u(t)=\sum_{j=0}^n \lambda_j\phi(|t-t_j|),
\end{eqnarray*}
where $t_j$'s are known as centers. The unknown parameters $\lambda_j$ are to be determined by the collocation method. Therefore, we get the following equations for the ODE
\begin{eqnarray}\label{eq1test1}
\sum_{j=0}^n \lambda_j~{_0 D_{_t}}^{\!\!3/2}\phi(|t_i-t_j|)
+\sum_{j=0}^n \lambda_j\phi(|t_i-t_j|)=f(t_i)
\end{eqnarray}
for $i=2,\ldots n-1$,and the following equations for the initial conditions
\begin{eqnarray}\label{eq1test2}
&&\sum_{j=0}^n \lambda_j~\phi(|t_1-t_j|)=0,\\\label{eq1test3}
&&\sum_{j=0}^n \lambda_j~\phi'(|t_1-t_j|)=0.
\end{eqnarray}
Then (\ref{eq1test1})-(\ref{eq1test3}) lead to the following system of equations:
\begin{eqnarray*}
\left[\begin{array}{l}
{_0 D_{_t}}^{\!\!3/2}\boldsymbol{\phi}+\boldsymbol{\phi}\\
\boldsymbol{\phi}_1\\
\boldsymbol{\phi}'_1
\end{array}\right]
\left[\boldsymbol{\lambda}\right]=
\left[\begin{array}{l}
\boldsymbol{F}\\
0\\
0
\end{array}\right].
\end{eqnarray*}
The necessary matrices and vectors are
\begin{eqnarray*}
&&\boldsymbol{\phi}=\left(\phi(|t_i-t_j|)\right)_{2\leq i\leq n-1,1\leq j\leq n},\\
&&{_0 D_{_t}}^{\!\!3/2}\boldsymbol{\phi}=\left({_0 D_{_t}}^{\!\!3/2}\phi(|t_i-t_j|)\right)_{2\leq i\leq n-1,1\leq j\leq n},\\
&&\boldsymbol{\phi}_1=\left(\phi(|t_1-t_j|)\right)_{1\leq j\leq n},\\
&&\boldsymbol{\phi}'_1=\left(\phi'(|t_1-t_j|)\right)_{1\leq j\leq n},\\
&&\boldsymbol{\lambda}={(\lambda_j,1\leq j\leq n)}^T,\\
&&\boldsymbol{F}={(f(t_i),2\leq i\leq n-1)}^T.
\end{eqnarray*}
Now, we take $501$ points in the interval $0\leq t\leq 50$ and work with the powers RBF $\phi
(x)=x^3$ with $c=10^{-4}$. 
The numerical solutions are plotted with different right-hand side functions $f(t)=1,$ $f(t)=t\mbox{e}^{-t},$ and $f(t)=\mbox{e}^{-t}\sin(0.2t)$ in Figures \ref{test1fig1}, \ref{test1fig2}, and \ref{test1fig4}, respectively. The results are in agreement with the results of \cite{podlubny:1999}.
\subsection{Test problem 2}
Consider the following Riesz fractional differential equation \cite{yang-et-al}:
\begin{eqnarray}
\begin{array}{l}\label{test2}
\frac{\partial u(x,t)}{\partial t}=-K_\alpha\frac{\partial^\alpha}{\partial{|x|}^\alpha}u(x,t),\qquad x\in[0,\pi],~t\in(0,T],\\
u(x,0)=u_0(x),\\
u(0,t)=u(\pi,t)=0,
\end{array}
\end{eqnarray}
where $u$ is, for example, a solute concentration and $K_\alpha$ represents the dispersion coefficient.
Let $x_i,~1\leq i\leq n$ be the equidistant discretization points in the interval $[0,\pi]$ such that $x_1=0$ and $x_n=\pi$. Now, we construct the Lagrange basis $L_1(x),\ldots,L_n(x)$ of the span of the functions $\phi(|x-x_j|),~1\leq j\leq n$ via solving the system 
$$\boldsymbol{L}(x)=\boldsymbol{\phi}(x)\boldsymbol{A}^{-1},$$
where
$$
\begin{array}{rcl}
\boldsymbol{L}(x)
&=&
(L_j(x))_{1\leq j\leq n},\\
\boldsymbol{\phi}(x)
&=&
(\phi(|x-x_j|)_{1\leq j\leq n},\\
\quad\boldsymbol{A}
&=&
\left(\phi(|x_i-x_j|)\right)_{1\leq i\leq n,1\leq j\leq n}.
\end{array} 
$$
If $\mathcal{L}$ is a differential operator, and if the RBF $\phi$ is sufficiently smooth to to allow application of $\mathcal{L}$, the required derivatives $\mathcal{L}L_j$ of the Lagrange basis $L_j$ come via solving$$(\mathcal{L}L)(x)=(\mathcal{L}\boldsymbol{\phi})(x)\boldsymbol{A}^{-1}.$$
Due to the standard Lagrange conditions, the zero boundary conditions at $x_0=0$ and $x=\pi$ are satisfied if we use the span of the functions $L_2,\ldots,L_{n-1}$ as our trial space. Then the approximate solution can be written as 
\begin{eqnarray*}
{u}(x,t)=\displaystyle\sum_{j=2}^{n-1}\beta_j(t)L_j(x),
\end{eqnarray*}
with the unknown vector $$\boldsymbol{\beta}(t)=(\beta_j(t),2\leq j\leq n-1).$$
We now write the PDE at a point $x_i,~2\leq i\leq n-1$ as follows:
\begin{eqnarray*}
\displaystyle\sum_{j=2}^{n-1}\beta_j'(t)L_j(x_i)=-K_\alpha\displaystyle\sum_{j=2}^{n-1}\beta_j(t)\frac{\partial^\alpha}{\partial{|x|}^\alpha}L_j(x_i).
\end{eqnarray*}
The initial conditions also provide
\begin{eqnarray*}
\beta_j(0)=u_0(x_j),\quad 2\leq j\leq n-1.
\end{eqnarray*}
Thus we get the following system of ODEs
\begin{eqnarray*}
\boldsymbol{\beta}'(t)=-K_\alpha\left(\boldsymbol{\frac{\partial^\alpha}{\partial{|x|}^\alpha}}\boldsymbol{L}*\boldsymbol{\beta}(t)\right),
\end{eqnarray*}
with the initial conditions
\begin{eqnarray*}
\boldsymbol{\beta}(0)=U_0,
\end{eqnarray*}
where
\begin{eqnarray*}
&&\boldsymbol{\frac{\partial^\alpha}{\partial{|x|}^\alpha}}\boldsymbol{L}={\left(\frac{\partial^\alpha}{\partial{|x|}^\alpha}L_j(x_i)\right)}_{2\leq i\leq n-1,2\leq j\leq n-1},\\
&&\boldsymbol{U}_0={(u_0(x_i),2\leq i\leq n-1)}^T.
\end{eqnarray*}
Now, consider problem (\ref{test2}) with the parameters $\alpha=1.8,$ $K_\alpha=0.25$, $T=0.4$, and $u_0(x)=x^2(\pi-x)$. The numerical solution is plotted by using the Gaussian RBF with $c=1$, and taking $101$ discretization points, in Figure \ref{a18t4}. 
In the second experiment, we use parameters $\alpha=1.5,$ $K_\alpha=-0.25$, $T=0.5$, and $u_0(x)=\sin(4x).$ The numerical solution is plotted by using the Gaussian RBF with $c=1$, and taking $101$ discretization points, in Figure \ref{a15t5}.
The results are in agreement with the results of \cite{yang-et-al}. It should be noted that using the multiquadric RBF with $\beta=1$ and $c=1$ also gives the same results.
\section{Conclusion}
The Riemann-Liouville fractional integral and derivative and also the Caputo fractional derivative of the five kinds of RBFs including the powers, Gaussian, multiquadric, Matern and thin-plate splines, in one dimension, are obtained. 
These formulas allow to use new fractional variations of numerical methods based on RBFs. Two examples of such techniques are given. The first one is a fractional ODE which is solved by the RBF collocation method and the second one is a fractional PDE which is solved by the method of lines based on the spatial trial spaces spanned by the Lagrange basis associated to RBFs.
\bibliographystyle{elsarticle-num}

\begin{thebibliography}{10}

\bibitem{benson-et-al}
D.A.~Benson, S.W.~Wheatcraft, and M.M.~Meerschaert.
\newblock Application of a fractional advection dispersion equation.
\newblock {\em Water Resour. Res.}, 36:1403--1412, 2000.

\bibitem{fix-roop}
G.J.~Fix and J.P.~Roop.
\newblock  Least square finite-element solution of a fractional order two-point boundary value problem.
\newblock {\em Comput. Math. Appl.}, 48:1017--1033, 2004.

\bibitem{frank-schaback:1998-1}
C.~Franke and R.~Schaback.
\newblock Convergence order estimates of meshless collocation methods using
  radial basis functions.
\newblock {\em Adv. in Comp. Math.}, 8:381--399, 1998.

\bibitem{frank-schaback:1998-2}
C.~Franke and R.~Schaback.
\newblock Solving partial differential equations by collocation using radial
  basis functions.
\newblock {\em Appl. Math. Comp.}, 93:73--82, 1998.

\bibitem{gorenflo-et-al}
R.~Gorenflo, F.~Mainardi, E.~Scalas, and M.~Raberto.
\newblock Fractional calculus and continuous-time finance. III, The diffusion limit. Mathematical finance.
\newblock {\em Trends in Math. Birkh\"{a}user, Basel.}, 171--180, 2001.

\bibitem{gao-sun:2011}
G.H.~Gao and Z.Z.~Sun.
\newblock A compact difference scheme for the fractional sub-diffusion equations.
\newblock {\em Comput. Phys.}, 230:586--595, 2001.

\bibitem{huang-et-al}
J.~Huang, N.~Ningming, and Y.~Tang.
\newblock A second order finite difference-spectral method for space fractional diffusion equation.

\bibitem{ishteva-et-al}
M.~Ishteva., R.~Scherer, and L.~Boyadjiev.
\newblock On the Caputo operator of fractional calculus and C-Laguerre functions.
\newblock {\em Math. Sci. Res. J.}, 9:161--170, 2005.

\bibitem{hon-et-al:2003-1}
Y.C.~Hon, R.~Schaback, and X.~Zhou.
\newblock An adaptive greedy algorithm for solving large radial basis function
  collocation problem.
\newblock {\em Numer. Algorithms.}, 32:13--25, 2003.

\bibitem{ilic-et-al:2005}
M.~Ilic, F.~Liu, I.~Turner, and V.~Anh.
\newblock Numerical approximation of a fractional-in-space diffusion equation (I).
\newblock {\em Fract. Calc. Appl. Anal.}, 8:323--341, 2005.

\bibitem{ilic-et-al:2006}
M.~Ilic, F.~Liu, I.~Turner, and V.~Anh.
\newblock Numerical approximation of a fractional-in-space diffusion equation (II) -with nonhomogeneous boundary conditions.
\newblock {\em Fract. Calc. Appl. Anal.}, 9:333--349, 2006.

\bibitem{kansa:1990-1}
E.J. Kansa.
\newblock scattered data approximation scheme with applications to
  computational fluid-dynamics, {I}: {S}olutions to parabolic, hyperbolic and
  elliptic partial differential equations.
\newblock {\em Comput. Math. Appl.}, 19:147--161, 1990.

\bibitem{koeller:1984}
R.C.~Koeller.
\newblock Application of fractional calculus to the theory of viscoelasticity.
\newblock {\em J. Appl. Mech.}, 51:229--307, 1984.

\bibitem{li-xu}
X.J.~Li and C.J.~Xu.
\newblock A space-time spectral method for the time fractional differential equation.
\newblock {\em SIAM J. Numer. Anal.}, 47:2108--2131, 2009.

\bibitem{lin-xu}
Y.M~Lin and C.J.~Xu.
\newblock Finite difference/spectral approximations for the time-fractional diffusion equation.
\newblock {\em J. Comput. Phys.}, 225:1533--1552, 2007.

\bibitem{li-et-al}
C.~Li, F.~Zeng, and F.~Liu.
\newblock Spectral approximation to the fractional integral and derivatives.
\newblock {\em Fract. Calc. Appl. Anal.}, 15:383--406, 2012.

\bibitem{mathai-haubold}
A.M.~Mathai and H.J.~Haubold.
\newblock {\em Special functions for applied sciences}.
\newblock Springer Science, 2008.

\bibitem{mathworld}
\newblock {\em http://mathworld.wolfram.com}

\bibitem{meerschaert-tadjeran}
M.M.~Meerschaert and C.~Tadjeran.
\newblock Finite difference approximations for fractional advection-diffusion equations.
\newblock {\em J. Comput. Appl. Math.}, 172:65--77, 2004.

\bibitem{mokhtari-mohammdi:2010-1}
R.~Mokhtari and M.~Mohammadi.
\newblock Numerical solution of {GRLW} equation using {S}inc-collocation
  method.
\newblock {\em Comput. Phys. Comm.}, 181:1266--1274, 2010.

\bibitem{mokhtari-mohseni}
R.~Mokhtari and M.~Mohseni.
\newblock  A meshless method for solving mKdV equation.
\newblock {\em  Comput. Phys. Comm.}, 183:1259--1268, 2012.

\bibitem{podlubny:1999}
I.~Podlubny.
\newblock {\em Fractional Differential Equations. An Introduction to Fractional Derivatives, Fractional Differential Equations, Some Methods of Their Solution
and Some of Their Applications}.
\newblock Academic Press, 1999.

\bibitem{raberto-et-al}
M.~Raberto, E.~Scalas, and F.~Mainardi.
\newblock Waiting-times and returns in high-frequency financial data: An empirical study.
\newblock {\em Phys. A.}, 314:749--755, 2002.

\bibitem{roop:2006}
J.P.~Roop.
\newblock Computational aspects of FEM approximation of fractional advection dispersion equations on bounded domains in $\mathbb{R}^2$.
\newblock {\em J. Comput. Appl. Math.}, 193:243--268, 2006.

\bibitem{schaback:2009-3}
R.~Schaback.
\newblock MATLAB Programming for Kernel-Based Methods.
\newblock Preprint G\"ottingen, 2009.

\bibitem{schaback-wendland:2006-1}
R.~Schaback and H.~Wendland.
\newblock Kernel techniques: From machine learning to meshless methods.
\newblock {\em Acta Numerica.}, 15:543--639, 2006.


\bibitem{schere-et-al}
R.~Schere, S.L.~Kalla, L.~ Boyadjievc, and B.~Al-Saqabi.
\newblock  Numerical treatment of fractional heat equations.
\newblock {\em Appl. Numer. Math.}, 58:1212--1223, 2008.

\bibitem{yang-et-al}
Q.~Yang., F.~Liu, and I.~Turner.
\newblock Numerical methods for fractional partial differential equations with Riesz space fractional derivatives.
\newblock {\em Appl. Math. Modelling.}, 34:200--218, 2010.

\bibitem{zhang:2009}
Y.~Zhang.
\newblock A finite difference method for fractional partial differential equation.
\newblock {\em Appl. Comput. Math.}, 215:524--529, 2009.

\end{thebibliography}


 \bigskip \smallskip

 \it

 \noindent
 Maryam Mohammadi\\
Kharazmi University, Tehran, Iran\\
e-mail: m.mohammadi@khu.ac.ir
~\\~\\
Robert Schaback\\
Institut f\"{u}r Numerische und 
Angewandte Mathematik\\
Universit\"{a}t G\"{o}ttingen\\ 
{Lotzestra\ss{}e} 
16-18, D-37073 G\"{o}ttingen, Germany\\[4pt]
e-mail: schaback@math.uni-goettingen.de
\newpage
\begin{table}
  \centering
   \caption{\footnotesize{Definition of some types of RBFs, where $\beta,$ $\nu,$ and $n$ are RBF parameters.}}\label{t1}
  \scriptsize{
\begin{tabular}{ccc}
  \hline
  Name & Definition\\
  \hline\\
  Gaussian & $\exp(-r^2/2)$\\\\
  Multiquadric & ${(1+r^2/2)}^{\beta/2}$\\\\
  Powers & $r^\beta$\\\\
  Matern/Sobolev & $K_\nu(r)r^\nu$\\\\
  Thin-plate splines & $r^{2n}\ln(r)$\\
 \hline
 \end{tabular}
}
\end{table}
\begin{figure}
 \center{ \includegraphics[width=250pt]{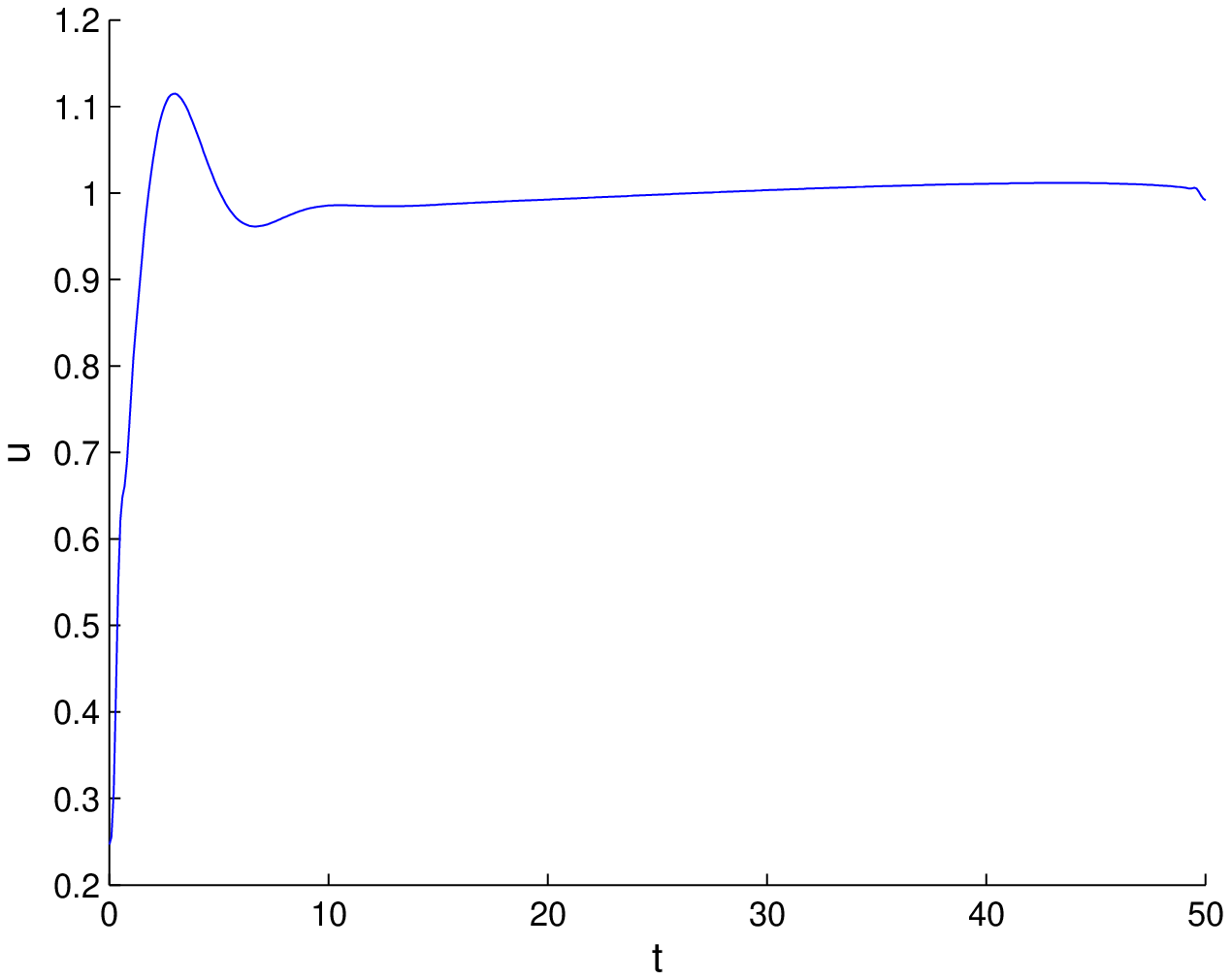}}
     \caption{Solution of the test problem 1 for $f(t)=1$.}\label{test1fig1}
\end{figure}
\begin{figure}
 \center{ \includegraphics[width=250pt]{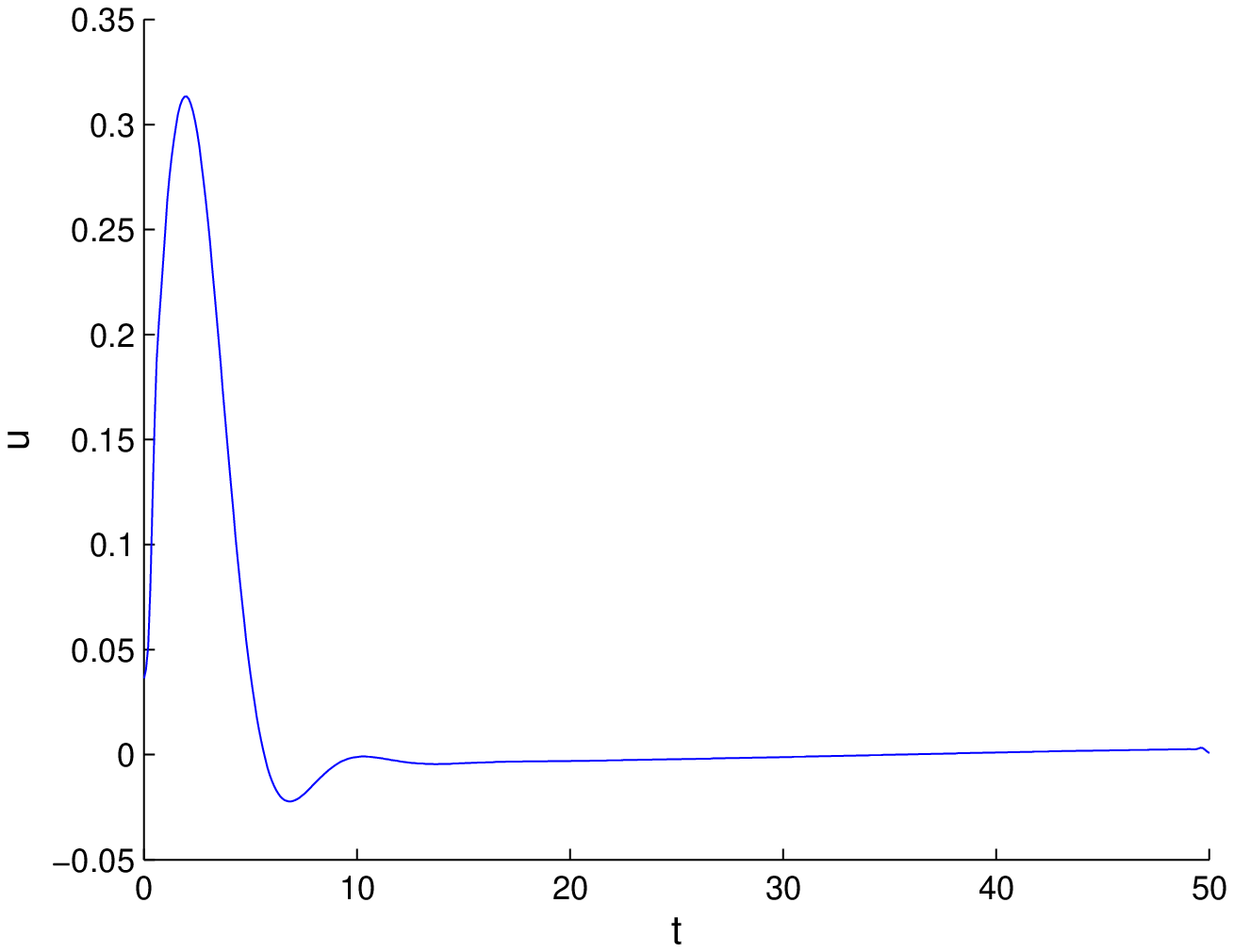}}
     \caption{Solution of the test problem 1 for $f(t)=t\mbox{e}^{-t}$.}\label{test1fig2}
\end{figure}
\begin{figure}
 \center{ \includegraphics[width=250pt]{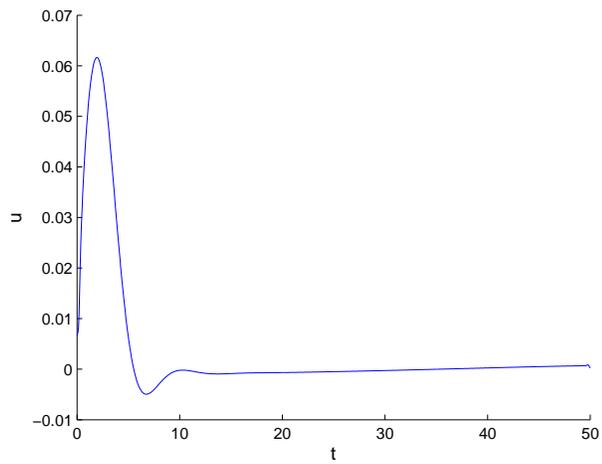}}
     \caption{Solution of the test problem 1 for $f(t)=\mbox{e}^{-t}\sin(0.2t)$.}\label{test1fig4}
\end{figure}
\begin{figure}
 \center{\includegraphics[width=250pt]{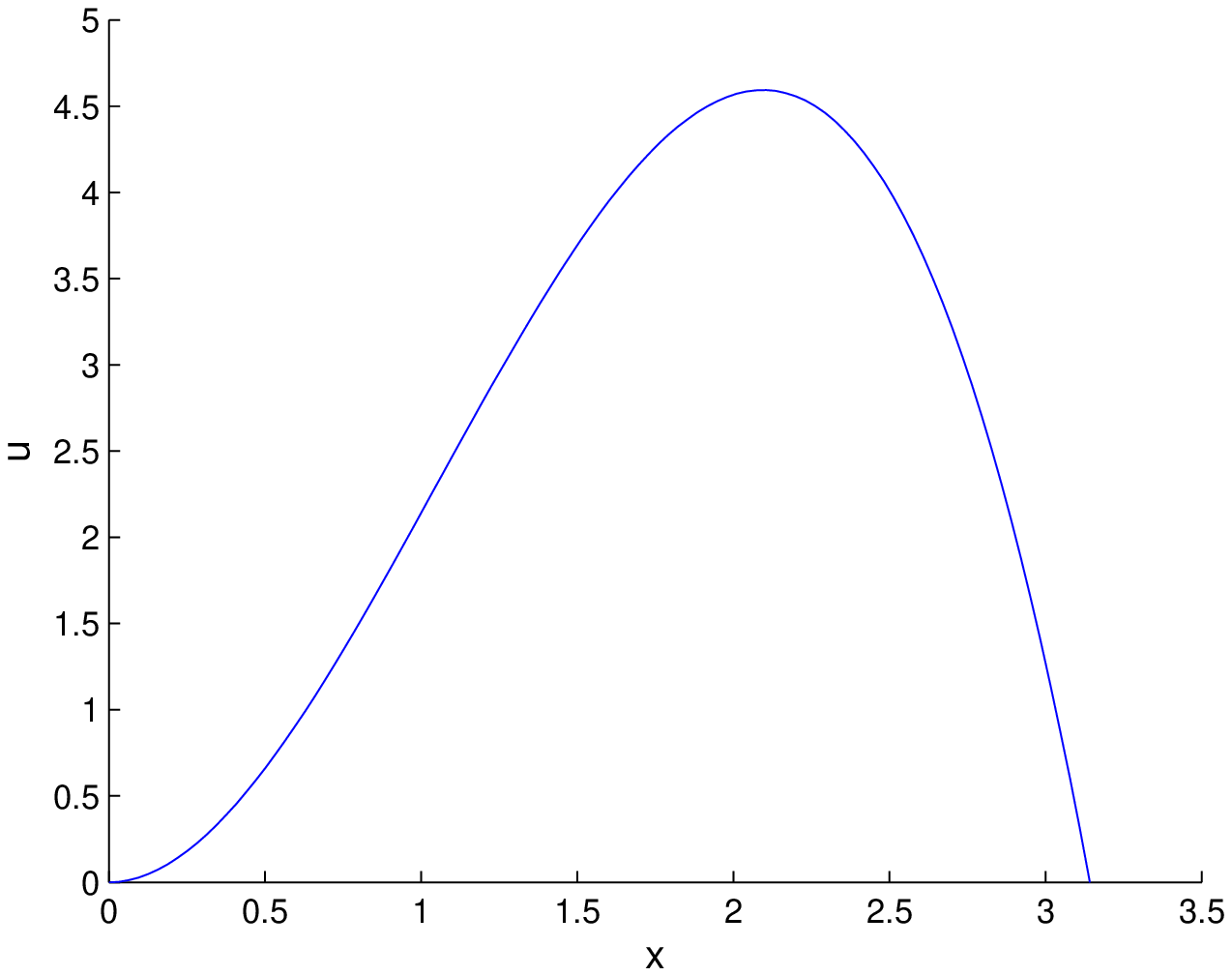}}
     \caption{Solution of the test problem 2 with  $\alpha=1.8,$ $K_\alpha=0.25,$ and $T=0.4$.}\label{a18t4}
\end{figure}
\begin{figure}
 \center{\includegraphics[width=250pt]{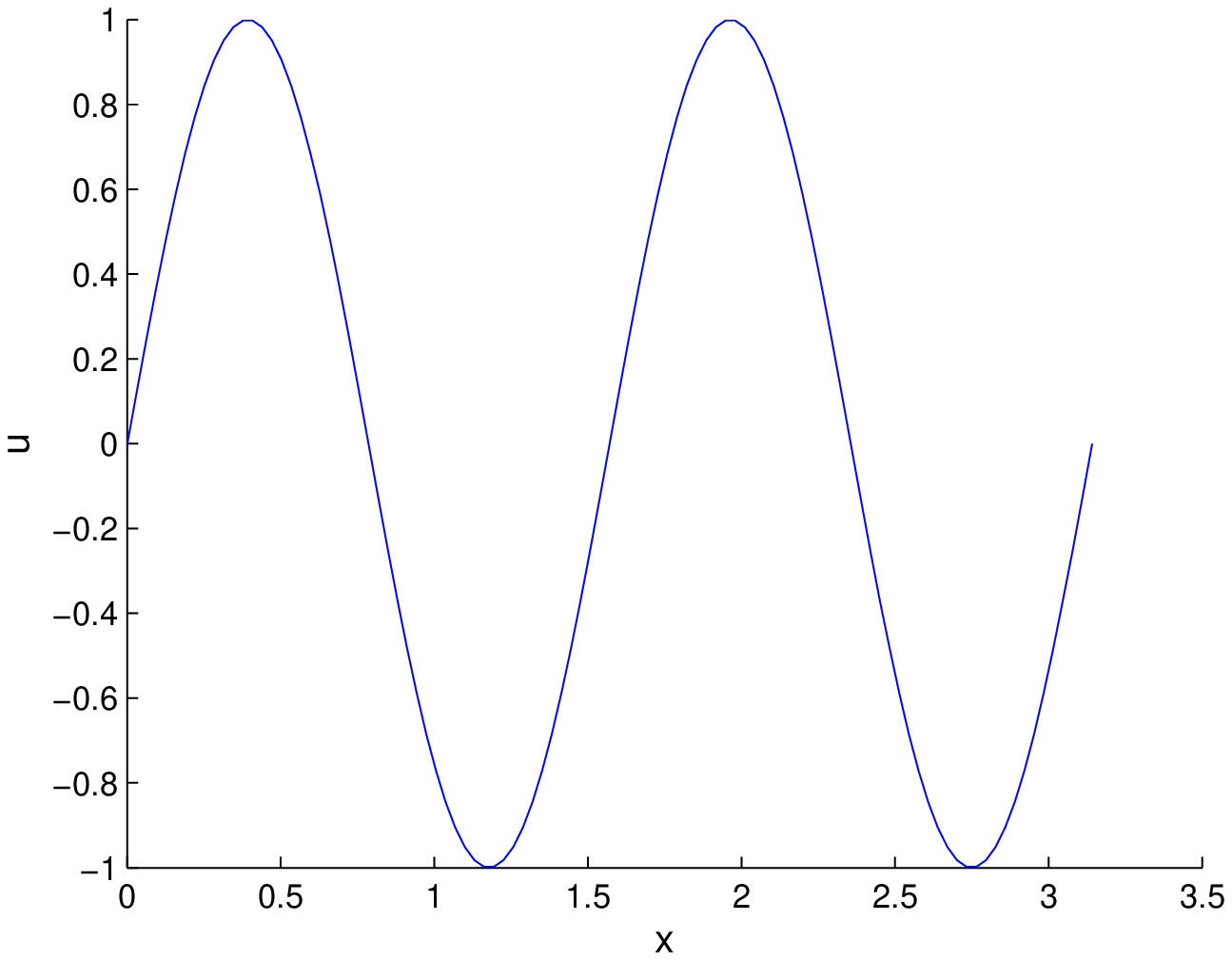}}
     \caption{Solution of the test problem 2 with  $\alpha=1.5,$ $K_\alpha=-0.25,$ and $T=0.5$.}\label{a15t5}
\end{figure}

\end{document}